\documentclass[11pt]{article}
\usepackage{enumerate}
\usepackage{amsmath}
\usepackage{amsthm}
\usepackage{amsfonts}
\usepackage{amssymb}
\usepackage[numbers]{natbib}
\usepackage{xcolor}
\usepackage{hyperref}

\usepackage{units}
\usepackage{float}
\usepackage{tikz}

\usepackage{graphicx}

\newcommand*\samethanks[1][\value{footnote}]{\footnotemark[#1]}
\usepackage{mathtools}

\usepackage{dsfont} 

\usepackage{fullpage}

\usepackage[capitalise]{cleveref} 
\crefname{equation}{}{}
\crefname{enumi}{}{}

\def\vp#1{}
\renewcommand{\vp}[1]{\footnote{\textcolor{green!40!black}{\textbf{VP: }#1}}}

\newcommand{\keywords}[1]
{
	{\small\textbf{Keywords:} #1}
}

\newtheorem{question}{Question}
\newtheorem{corollary}[question]{Corollary}
\newtheorem{problem}[question]{Problem}
\newtheorem{conjecture}[question]{Conjecture}
\newtheorem{theorem}[question]{Theorem}

\newtheorem{proposition}[question]{Proposition}
\newtheorem{lemma}[question]{Lemma}
\newtheorem{remark}[question]{Remark}
\newtheorem{claim}[question]{Claim}
\newtheorem{definition}[question]{Definition}
\newtheorem{construction}[question]{Construction}
\numberwithin{question}{section}
\numberwithin{equation}{section}

\title{Rainbow variations on a theme by Mantel:\\
	extremal problems for Gallai colouring templates}
\author{Victor Falgas-Ravry\thanks{Institutionen f\"or Matematik och Matematisk Statistik, Ume{\aa} Universitet, Sweden. Emails: \texttt{victor.falgas-ravry}, \texttt{klas.markstrom}, \texttt{eero.raty} \texttt{@umu.se}. } \and Klas Markstr\"om\samethanks \and Eero R\"aty\samethanks}
\begin{document}
	\maketitle	
\begin{abstract}
Let $\mathbf{G}:=(G_1, G_2, G_3)$ be a triple of graphs on the same vertex set $V$ of size $n$. A rainbow triangle in $\mathbf{G}$ is a triple of edges $(e_1, e_2, e_3)$ with $e_i\in G_i$ for each $i$ and $\{e_1, e_2, e_3\}$ forming a triangle in $V$. The triples $\mathbf{G}$ not containing rainbow triangles, also known as Gallai colouring templates, are a widely studied class of objects in extremal combinatorics.

In the present work, we fully determine the set of edge densities $(\alpha_1, \alpha_2, \alpha_3)$ such that if $\vert E(G_i)\vert> \alpha_i n^2$ for each $i$ and $n$ is sufficiently large, then $\mathbf{G}$ must contain a rainbow triangle. This resolves a problem raised by Aharoni, DeVos, de la Maza, Montejanos and \v{S}\'amal, generalises several previous results on extremal Gallai colouring templates, and proves a recent conjecture of Frankl, Gy\H{o}ri, He, Lv, Salia, Tompkins, Varga and Zhu. 
\end{abstract}	

\keywords{extremal graph theory, rainbow triangles, Gallai colourings, Mantel's theorem}

\section{Introduction}
Mantel's Theorem from 1907~\cite{Mantel07} is one of the foundational results in extremal graph theory. It asserts that a triangle-free graph $G$ on $n$ vertices has at most $\lfloor \frac{n^2}{4}\rfloor$ edges, with equality if and only if $G$ is (isomorphic to) the complete balanced bipartite graph $T_2(n)$. While the proof of Mantel's theorem is a simple combinatorial exercise, triangle-free graphs act as a kind of theoretical lodestone in extremal combinatorics: many important extremal tools or problems are first developed or studied in the context of triangle-free graphs. One may think, for example, of results on the independence number of triangle-free graphs~\cite{Shearer83}, the chromatic threshold phenomenon~\cite{AndrasfaiErdosSos74, Thomassen02}, the triangle removal lemma~\cite{RuzsaSzemeredi76}, and on random~\cite{DeMarcoKahn15} and tripartite~\cite{BondyShenThomasseThomassen06} versions of Mantel's theorem.

In this paper we will consider a rainbow variation on Mantel's triangle-free theme, which was first introduced by Gallai in 1967.  Fix an $n$-set $V$ and some integer $r \geq 2$. 
\begin{definition}\label{def: template}[Colouring templates, colourings]
	An $r$-colouring template on $V$ is an $r$-tuple $\mathbf{G}^{(r)}=(G_1, G_2, \ldots , G_r)$, where each of the $G_i$ is a graph on $V$. Whenever $r$ is clear from context, we omit the superscript $r$ and write $\mathbf{G}$ for $\mathbf{G}^{(r)}$.

	  An $r$-coloured graph $(H, c)$ is a graph $H=(V(H), E(H))$ together with an $r$-colouring of its edges $c: \ E(H)\rightarrow \{1,2, \ldots, r\}$.  (Note that an $r$-coloured graph may be identified with an $r$-colouring template where the colour classes $G_i$, $1\leq i\leq r$, are pairwise edge-disjoint.)
\end{definition}  
	 \begin{definition}[Coloured and rainbow subgraphs]
	 	  Given an $r$-coloured graph $(H,c)$, we say that an $r$-colouring template $\mathbf{G}^{(r)}$ on a vertex set $V$ contains a copy of $(H,c)$ as a subgraph if there is an injection $f: \ V(H)\rightarrow  V$ such that for each edge $e=\{x,y\} \in E(H)$ we have $\{f(x),f(y)\}\in G_{c(e)}$.  Further, given  a graph $H$, we say that $\mathbf{G}$ contains a rainbow copy of $H$ if $\mathbf{G}$ contains $(H,c)$ for some $r$-colouring $c: \ E(H)\rightarrow \{1,2, \ldots, r\}$  assigning distinct colours to distinct edges.
\end{definition}
\noindent Gallai~\cite{Gallai67} initiated the study of $r$-colourings with no rainbow triangles, proving a structure theorem that was subsequently re-discovered and extended by a number of other researchers~\cite{CameronEdmonds97,GyarfasSimonyi04}; in honour of his pioneering contributions to the area, $r$-coloured graphs containing no rainbow triangle are known as \emph{Gallai colourings}. We accordingly refer to $r$-colouring templates not containing a rainbow copy of $K_3$ as \emph{Gallai colouring templates}.

Gallai colourings have been extensively studied. For instance, there are connections between Gallai colourings and information theory~\cite{KornerSimony00}, and a considerable interest in counting the number of Gallai colourings and characterising their typical structure~\cite{BaloghLi19, BenevidesHoppenSampaio17,FROConnellUzzell19}. A large body of work has been dedicated to research on Gallai colourings from a Ramsey-theoretic perspective, giving rise to `Gallai--Ramsey theory' ---  see the dynamic survey~\cite{FujitaMagnantOzeki14} devoted to the area.

In this paper, we focus instead on Tur\'an-style questions for Gallai colouring templates. One of the first results of this kind was obtained by Keevash, Saks, Sudakov and Verstra\"ete~\cite{KeevashSaksSudakovVerstraete04}, who determined the arithmetic mean of the size of the colour classes $G_1, G_2, \ldots G_r$ required to guarantee the existence of a rainbow $K_3$ in an $r$-colouring template. As a special case of more general results on rainbow cliques, they proved the following~\cite[Theorem~1.2]{KeevashSaksSudakovVerstraete04}:
\begin{theorem}~\label{theorem: average density forcing rainbow $K_3$}[Keevash, Saks, Sudakov, Verstra\"ete]
	If $\mathbf{G}$ is a Gallai $r$-colouring template on $n$ vertices for $n$ sufficiently large, then 
	\[\frac{1}{r}\sum_{i=1}^r \vert E(G_i)\vert \leq \left\{\begin{array}{ll}
	\frac{2}{3}\binom{n}{2} &\textrm{ if }r=3,\\
	\left\lfloor \frac{n^2}{4}\right\rfloor & \textrm{ if }r\geq 4,
	\end{array} \right.\]	
	and these upper bounds are best possible.
\end{theorem}
\noindent The lower bound constructions in Theorem~\ref{theorem: average density forcing rainbow $K_3$} are the trivial ones: for $r\geq 4$, one takes $G_1=G_2=\ldots =G_r=T_2(n)$, while for $r=3$ one takes $G_1=G_2=K_n$ and lets $G_3$ be the empty graph. Given that this latter construction features  an empty colour class, it is natural to ask how the bound in Theorem~\ref{theorem: average density forcing rainbow $K_3$} changes in the $r=3$ case if one requires all three of the colour classes $G_1$, $G_2$ and $G_3$ to be large. This question was first posed by Diwan and Mubayi in a 2006 manuscript~\cite{DiwanMubayi06}: what is the least $\alpha>0$ such that for all $n$ sufficiently large, every $3$-colouring template $\mathbf{G}$ on an $n$-set $V$ with $\min\{\vert E(G_i)\vert: \ 1\leq i\leq 3\}>\alpha n^2$ contains a rainbow triangle? In other words, how large do you need the smallest of the three colour classes to be in order to guarantee the existence of a rainbow triangle?

Magnant~\cite[Theorem 5]{Magnant15} answered this question in 2015 under the assumption that the  union of the colour classes $G_i$, $1\leq i\leq 3$,  covers all pairs in $V$. This assumption may seem natural, insofar as one seeks to make all colour classes large, but it also introduces some very strong restrictions on the colouring template $\mathbf{G}$. Indeed, if $\{x,y\}\in E(G_i)\cap E(G_j)$ and $\{x,z\}\in E(G_i)\cap E(G_k)$ for some distinct indices $1\leq i,j,k \leq 3$, then if the edge $\{y,z\}$ belongs to any of the three colour classes we have a rainbow triangle. Thus Magnant's assumption rules out any vertex being adjacent to two `bi-chromatic edges' with distinct colour pairs. In a 2020 paper, Aharoni, DeVos, de la Maza, Montejanos and \v{S}\'amal~\cite[Theorem 1.2]{AharoniDeVosdelaMazaMontejanoSamalin20} did away with Magnant's technical assumption and answered the question of Diwan and Mubayi in full.  Let $\tau:=\frac{4-\sqrt{7}}{9}$.
\begin{theorem}\label{theorem: min colour density forcing rainbow K3}[Aharoni, DeVos, de la Maza, Montejano and \v{S}\'amal]
For all $n$ sufficiently large, any $n$-vertex $3$-colouring template $\mathbf{G}$ satisfying
\[\min\left\{\vert E(G_1)\vert, \vert E(G_2)\vert, \vert E(G_3)\vert\right\}>\frac{1+\tau^2}{4}n^2\]
contains a rainbow triangle.	
\end{theorem}
\noindent Moreover, the lower bound in Theorem~\ref{theorem: min colour density forcing rainbow K3} is tight up to a $O(n)$ additive term, as can be seen by considering the following family of constructions. Set $\left[n\right]:=\{1,2\ldots, n\}$, and write $S^{\left(2\right)}$ for the collection of unordered pairs of elements from a set $S$.
\begin{construction}[$\mathbf{F}(a,b,c)$-templates]\label{construction: F(a,b,c)}
	Let $a$, $b$ and $c$ be non-negative integers with $a+b+c=n$. Arbitrarily partition $\left[n\right]$ as $\left[n\right]=A\sqcup B\sqcup C$, with $\left|A\right|=a$, $\left|B\right|=b$ and $\left|C\right|=c$. Define graphs $F_1$, $F_2$ and $F_3$ on the vertex set $\left[n\right]$
	by setting 
	\begin{align*}
	F_{1}:=A^{\left(2\right)}\cup B^{\left(2\right)},&& F_{2}:=A^{\left(2\right)}\cup C^{\left(2\right)}, \textrm{ and } && F_{3}:=\left[n\right]^{(2)}\setminus A^{\left(2\right)}.\end{align*}
	Write $\mathbf{F}=\mathbf{F}(a,b,c)$ for (any instance of) the $n$-vertex $3$-colouring template $(F_1, F_2, F_3)$.	
\end{construction}
\noindent See Figure~\ref{figure: F and G} for a picture of the $3$-colouring template $\mathbf{F}(a,b,c)$. It is readily checked that $\mathbf{F}$ is rainbow $K_3$-free, and that setting $b=c=\lceil \tau n\rceil$ and $a=n-2\lceil \tau n\rceil$ we have that all three colour classes $F_1$, $F_2$ and $F_3$ contain $\frac{1+\tau^2}{4}n^2 +O(n)$ edges.

The authors of~\cite{AharoniDeVosdelaMazaMontejanoSamalin20} suggested the more general problem of determining which triples of edge densities $(\alpha_1, \alpha_2, \alpha_3)$ force a rainbow triangle~\cite[Problem 1.3]{AharoniDeVosdelaMazaMontejanoSamalin20}.
\begin{definition}[Forcing triple]
	A triple $(\alpha_1, \alpha_2, \alpha_3)\in [0,1]^3$ is a forcing triple if for all $n$ sufficiently large, every $n$-vertex $3$-colouring template $\mathbf{G}$ satisfying $e(G_i)>\min\left(\frac{\alpha_i}{2} n^2, \binom{n}{2}-1\right)$ for $i\in\{1,2,3\}$ must contain a rainbow triangle.
\end{definition}
\noindent In this terminology\footnote{In this paper we use the normalisation term $n^2/2$ instead of the $n^2$ term used in~\cite{AharoniDeVosdelaMazaMontejanoSamalin20} as most of our argument will be written in terms of binomial coefficients $\binom{n}{2}$.}, the authors of~\cite{AharoniDeVosdelaMazaMontejanoSamalin20} proposed the following generalisation of Diwan and Mubayi's question:
\begin{problem}\label{problem: which densities are forcing}
	Determine the set of forcing triples.
\end{problem}
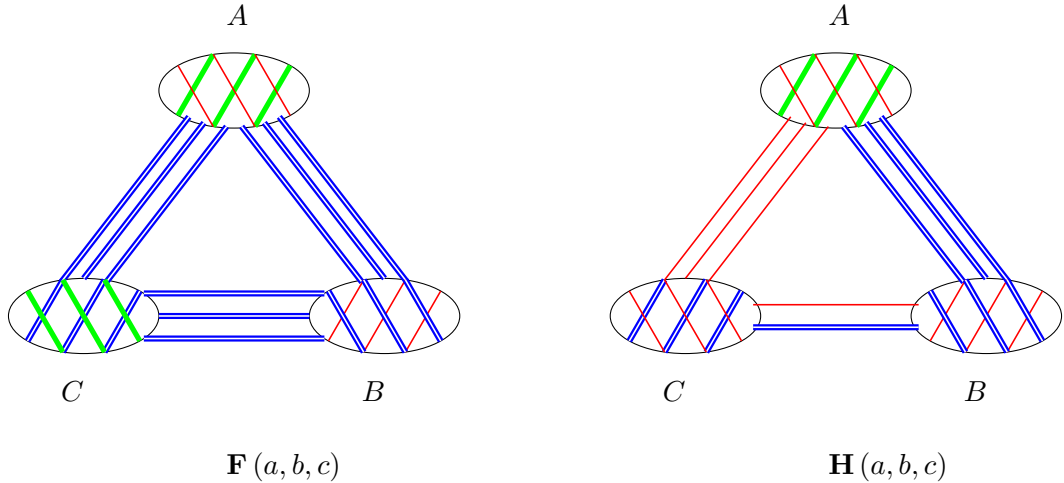
\begin{figure}[H]\label{figure: F and G}
	
	\centering
	\caption{The Gallai colouring templates $\mathbf{F}\left(a,b,c\right)$ and $\mathbf{H}\left(a,b,c\right)$ with red (thin lines), green (thick lines) and blue (doubled lines) representing edges in colours $1$, $2$ and $3$ respectively.}
	\begin{tikzpicture}
	\draw (0,0) ellipse (1cm and 0.5cm);
	\draw (4,0) ellipse (1cm and 0.5cm);
	\draw (2,3) ellipse (1cm and 0.5cm);
	
	\draw [line width = 0.3 mm, blue,double] (0.8, 0.3) -- (3.2, 0.3);
	\draw [line width = 0.3 mm, blue,double] (0.8, -0.3) -- (3.2, -0.3);
	\draw [line width = 0.3 mm, blue,double] (1, 0) -- (3, 0);
	
	\draw [line width = 0.3 mm, blue,double] (-0.3, 0.45) -- (1.4, 2.65);
	\draw [line width = 0.3 mm, blue,double] (0, 0.5) -- (1.6, 1.6*1.294 + 0.5);
	\draw [line width = 0.3 mm, blue,double] (0.3, 0.45) -- (0.3 + 1.6, 1.6*1.294 + 0.45);
	
	\draw [line width = 0.3 mm, blue,double] (4.3, 0.45) -- (2.6, 2.65);
	\draw [line width = 0.3 mm, blue,double] (4, 0.5) -- (2.4, 2.5704);
	\draw [line width = 0.3 mm, blue,double] (3.7, 0.45) -- (2.1, 2.5204);
	
	\draw [line width = 0.3 mm, blue,double] (-0.7467, -0.3326) -- (-0.2773,0.48);
	\draw [line width = 0.3 mm, blue,double] (-0.2773,-0.48) -- (0.2773,0.48);
	\draw [line width = 0.3 mm, blue,double] (0.2773,-0.48) -- (0.7467,0.3326);

	\draw [line width = 0.7 mm, green] (-0.7467, 0.3326) -- (-0.2773,-0.48);
	\draw [line width = 0.7 mm, green] (-0.2773,0.48) -- (0.2773,-0.48);
	\draw [line width = 0.7 mm, green] (0.2773,0.48) -- (0.7467,-0.3326);

	\draw [line width = 0.7 mm, red,semithick] (4-0.7467, -0.3326) -- (4-0.2773,0.48);
	\draw [line width = 0.7 mm, red,semithick] (4-0.2773,-0.48) -- (4+0.2773,0.48);
	\draw [line width = 0.7 mm, red,semithick] (4+0.2773,-0.48) -- (4+0.7467,0.3326);

	\draw [line width = 0.3 mm, blue,double] (4-0.7467, 0.3326) -- (4-0.2773,-0.48);
	\draw [line width = 0.3 mm, blue,double] (4-0.2773,0.48) -- (4+0.2773,-0.48);
	\draw [line width = 0.3 mm, blue,double] (4+0.2773,0.48) -- (4+0.7467,-0.3326);

	\draw [line width = 0.7 mm, green] (2-0.7467, 3-0.3326) -- (2-0.2773,3+0.48);
	\draw [line width = 0.7 mm, green] (2-0.2773,3-0.48) -- (2+0.2773,3+0.48);
	\draw [line width = 0.7 mm, green] (2+0.2773,3-0.48) -- (2+0.7467,3+0.3326);

	\draw [line width = 0.7 mm, red,semithick] (2-0.7467, 3+0.3326) -- (2-0.2773,3-0.48);
	\draw [line width = 0.7 mm, red,semithick] (2-0.2773,3+0.48) -- (2+0.2773,3-0.48);
	\draw [line width = 0.7 mm, red,semithick] (2+0.2773,3+0.48) -- (2+0.7467,3-0.3326);
	
	\draw (8,0) ellipse (1cm and 0.5cm);
	\draw (12,0) ellipse (1cm and 0.5cm);
	\draw (10,3) ellipse (1cm and 0.5cm);
	
	\draw [line width = 0.7 mm, red,semithick] (8+0.9, 0.15) -- (8+3.1, 0.15);
	\draw [line width = 0.3 mm, blue,double] (8+0.9, -0.15) -- (8+3.1, -0.15);
	
	\draw [line width = 0.7 mm, red,semithick] (8-0.3, 0.45) -- (8+1.4, 2.65);
	\draw [line width = 0.7 mm, red,semithick] (8+0, 0.5) -- (8+1.6, 1.6*1.294 + 0.5);
	\draw [line width = 0.7 mm, red,semithick] (8+0.3, 0.45) -- (8+0.3 + 1.6, 1.6*1.294 + 0.45);
	
	\draw [line width = 0.3 mm, blue,double] (8+4.3, 0.45) -- (8+2.6, 2.65);
	\draw [line width = 0.3 mm, blue,double] (8+4, 0.5) -- (8+2.4, 2.5704);
	\draw [line width = 0.3 mm, blue,double] (8+3.7, 0.45) -- (8+2.1, 2.5204);
	
	\draw [line width = 0.3 mm, blue,double] (8-0.7467, -0.3326) -- (8-0.2773,0.48);
	\draw [line width = 0.3 mm, blue,double] (8-0.2773,-0.48) -- (8+0.2773,0.48);
	\draw [line width = 0.3 mm, blue,double] (8+0.2773,-0.48) -- (8+0.7467,0.3326);

	\draw [line width = 0.7 mm, red,semithick] (8-0.7467, 0.3326) -- (8-0.2773,-0.48);
	\draw [line width = 0.7 mm, red,semithick] (8-0.2773,0.48) -- (8+0.2773,-0.48);
	\draw [line width = 0.7 mm, red,semithick] (8+0.2773,0.48) -- (8+0.7467,-0.3326);

	\draw [line width = 0.7 mm, red,semithick] (8+4-0.7467, -0.3326) -- (8+4-0.2773,0.48);
	\draw [line width = 0.7 mm, red,semithick] (8+4-0.2773,-0.48) -- (8+4+0.2773,0.48);
	\draw [line width = 0.7 mm, red,semithick] (8+4+0.2773,-0.48) -- (8+4+0.7467,0.3326);

	\draw [line width = 0.3 mm, blue,double] (8+4-0.7467, 0.3326) -- (8+4-0.2773,-0.48);
	\draw [line width = 0.3 mm, blue,double] (8+4-0.2773,0.48) -- (8+4+0.2773,-0.48);
	\draw [line width = 0.3 mm, blue,double] (8+4+0.2773,0.48) -- (8+4+0.7467,-0.3326);

	\draw [line width = 0.7mm, green] (8+2-0.7467, 3-0.3326) -- (8+2-0.2773,3+0.48);
	\draw [line width = 0.7mm, green] (8+2-0.2773,3-0.48) -- (8+2+0.2773,3+0.48);
	\draw [line width = 0.7mm, green] (8+2+0.2773,3-0.48) -- (8+2+0.7467,3+0.3326);

	\draw [line width = 0.7mm, red,semithick] (8+2-0.7467, 3+0.3326) -- (8+2-0.2773,3-0.48);
	\draw [line width = 0.7mm, red,semithick] (8+2-0.2773,3+0.48) -- (8+2+0.2773,3-0.48);
	\draw [line width = 0.7mm, red,semithick] (8+2+0.2773,3+0.48) -- (8+2+0.7467,3-0.3326);
	
	\node[text width=1cm] at (2.4,4){$A$};
	\node[text width=1cm] at (0.2,-1){$C$};
	\node[text width=1cm] at (4.2,-1){$B$};

	\node[text width=1cm] at (8+2.4,4){$A$};
	\node[text width=1cm] at (8+0.2,-1){$C$};
	\node[text width=1cm] at (8+4.2,-1){$B$};
	
	\node[text width=1cm] at (2.4, -2){$\mathbf{F}\left(a,b,c\right)$};
	\node[text width=1cm] at (10.4, -2){$\mathbf{H}\left(a,b,c\right)$};
	
	\end{tikzpicture}
\end{figure}

Recently Frankl~\cite[Theorem 1.4]{Frankl22} gave a new proof of Theorem~\ref{theorem: average density forcing rainbow $K_3$} on the maximum arithmetic mean of the sizes of the colour classes in a Gallai $r$-colouring template, and raised the problem of maximising the geometric mean of the sizes of the colour classes for such templates in the case\footnote{For $r\geq 4$, the AM--GM inequality together with Theorem~\ref{theorem: average density forcing rainbow $K_3$} immediately implies the geometric mean of the colour classes in a Gallai $r$-colouring template is at most $\lfloor\frac{n^2}{4}\rfloor$ for all $n$ sufficiently large, so the case $r=3$ is the only one for which this question is open.}  $r=3$. This can be viewed as a different way of forcing all three colour classes $G_1$, $G_2$ and $G_3$ to be (reasonably) large, and of moving away from the extremal construction where two of the colour classes are complete and the third is empty.

Frankl proved an upper bound of $\left\lfloor \frac{n^2}{4}\right\rfloor$ on this geometric mean under the assumption that the colour classes were nested~\cite[Theorem 1.5]{Frankl22}. This result is tight under the nestedness assumption: a lower bound construction is obtained by taking three identical copies of $T_2(n)$ for the three colour classes.

Frankl conjectured that his upper bound on the geometric mean was tight in general, without the nestedness assumption on the colour classes~\cite[Conjecture 3]{Frankl22}. This was subsequently disproved by Frankl, Gy\H{o}ri, He, Lv, Salia, Tompkins, Varga and Zhu, who provided a different construction, which they conjectured~\cite[Conjecture 2]{FranklGyoriHeLvSaliaTompkins22} maximises the geometric mean of the sizes of the colour classes in a Gallai $3$-colouring template. Their construction turns out to be a special case of a more general construction that will play a key role in this paper, and which we define below. Write $(S, T)^{(2)}$ for the collection of unordered pairs taking one vertex from each of $S$ and $T$.
\begin{construction}[$\mathbf{H}(a,b,c)$-templates]\label{construction: H(a,b,c)}
	Let $a$, $b$ and $c$ be non-negative integers with $a+b+c=n$. Arbitrarily partition $\left[n\right]$ as $\left[n\right]=A\sqcup B\sqcup C$, with $\left|A\right|=a$, $\left|B\right|=b$ and $\left|C\right|=c$. Define graphs $H_1$, $H_2$
	and $H_3$ on the vertex set $\left[n\right]$
	by setting 
	\begin{align*}
	H_{1}:=A^{\left(2\right)}\cup\left(B\cup C\right)^{\left(2\right)}\cup\left(A,C\right)^{\left(2\right)},&& H_{2}:=A^{\left(2\right)}, \textrm{ and } && H_{3}:=\left(B\cup C\right)^{\left(2\right)}\cup\left(A,B\right)^{\left(2\right)}.\end{align*}
	Write $\mathbf{H}=\mathbf{H}(a,b,c)$ for (any instance of) the $n$-vertex $3$-colouring template $(H_1, H_2, H_3)$.	
\end{construction}
\noindent See Figure~\ref{figure: F and G} for a picture of the $3$-colouring template $\mathbf{H}(a,b,c)$. The special case $c=0$, $b=n-a$ corresponds to the constuction provided by the authors of~\cite{FranklGyoriHeLvSaliaTompkins22}. It is readily checked that $\mathbf{H}$ is rainbow $K_3$-free. Let $\upsilon$ denote the value of $x\in [0,1]$ maximising the value of the function 
\[h: \ x\mapsto \left(x^2+(1-x)^2\right)x^2\left(1-x^2\right).\]
\noindent The value of $\upsilon$ may be computed explicitly, though the exact form is not pleasant. Numerically, we have $\upsilon\approx 0.7927$ and $h(\upsilon)\approx 0.1568$. Setting $a=\lceil \upsilon n\rceil$, $b=n-a$ and $c=0$, we have that 
\[\Bigl(\vert E(H_1)\vert \cdot \vert E(H_2)\vert \cdot \vert E(H_3)\vert\Bigr)^{\frac{1}{3}}=\left(h(\upsilon)+o(1)\right)^{\frac{1}{3}} \frac{n^2}{2}= \left(0.5392+o(1)\right) \frac{n^2}{2},\]
which is significantly larger than $\left\lfloor \frac{n^2}{4}\right\rfloor$ for all $n$ sufficiently large. Thus, as noted by the authors of~\cite{FranklGyoriHeLvSaliaTompkins22}, the Gallai $3$-colouring template $\mathbf{H}$ for these values of $a$, $b$ and $c$ provides a counterexample to the aforementioned conjecture of Frankl.  However they conjectured~\cite[Conjecture 2]{FranklGyoriHeLvSaliaTompkins22} that asymptotically one could not do better than the $\mathbf{H}(\lceil \upsilon n\rceil, n-\lceil \upsilon n \rceil, 0)$ Gallai $3$-colouring template:
\begin{conjecture}[Frankl, Gy\H{o}ri, He, Lv, Salia, Tompkins, Varga and Zhu]\label{conjecture: product extremal}
Let $\mathbf{G}$ be a Gallai $3$-colouring template on $n$ vertices. Then
\[\Bigl( \vert E(G_1)\vert \cdot \vert E(G_2)\vert \cdot \vert E(G_3)\vert\Bigr)^{\frac{1}{3}} \leq \left(h(\upsilon)+o(1)\right)^{\frac{1}{3}} \binom{n}{2}.\]	
\end{conjecture}
\noindent The authors of~\cite{FranklGyoriHeLvSaliaTompkins22} proved their conjecture under the assumption that the union of the colour classes covers the entire graph~\cite[Theorem 2]{FranklGyoriHeLvSaliaTompkins22}  --- the same assumption made earlier by Magnant, and which, as we remarked above, is both natural and highly restrictive in terms of the possible structure of $\mathbf{G}$.

\subsection{Results}
In the present work we fully resolve Problem~\ref{problem: which densities are forcing}. This asymptotically generalises previous Tur\'an-type results for Gallai $3$-colouring templates (Theorem~\ref{theorem: average density forcing rainbow $K_3$} and Theorem~\ref{theorem: min colour density forcing rainbow K3}), and settles Conjecture~\ref{conjecture: product extremal} in the affirmative. To state our result, we must define three regions in $[0,1]^2$.

\begin{definition}
 Let $\mathcal{R}_1$ denote the collection of $(\alpha_1, \alpha_2)\in [0,1]^2$ satisfying:
 \begin{align*}
 \max\left(1-\alpha_2, \frac{1+\tau^2}{2}, \alpha_2 \right)\leq \alpha_1 \leq 1-2\sqrt{\alpha_2}+2\alpha_2.
 \end{align*}
For $(\alpha_1, \alpha_2)\in \mathcal{R}_1$ there exists\footnote{The existence of this pair is proved in Proposition~\ref{prop: existence and uniqueness of the x,y}.} a unique pair $(x,y)$ of non-negative real numbers such that $x\geq\frac{1}{2}$, $x+y\leq 1$ and $\alpha_1=x^2+y^2$, $\alpha_2=x^2+(1-x-y)^2$; we refer to this pair as the \emph{canonical representation} of $(\alpha_1, \alpha_2)\in \mathcal{R}_1$.  We define $\mathcal{R}'_1$ to be the collection of $(\alpha_1, \alpha_2)\in \mathcal{R}_1$ whose canonical representation $(x,y)$ satisfies $2x^2+(1-x-y)^2\geq 1$.
\end{definition}
\begin{remark}
We can in principle compute the canonical pair $(x,y)$ explicitly from $(\alpha_1, \alpha_2)$: setting $y=\sqrt{\alpha_1-x^2}$, we need $x$ to be a solution in $[\frac{1}{2},\sqrt{\alpha_1}]$ to the equation
\begin{align}\label{eq: computing canonical pair}
	\alpha_1-\alpha_2 = (1-x)\left(x+2 \sqrt{\alpha_1 -x^2} - 1\right)\end{align}
while satisfying $x+\sqrt{\alpha_1-x^2}\leq 1$.	Now, \eqref{eq: computing canonical pair} can be rewritten as a quartic equation
	\[\left(\alpha_1-\alpha_2+(1-x)^2\right)^2 =4(1-x)^2(\alpha_1-x^2) ,\]
	whose solutions can be computed explicitly via radicals in terms of $\alpha_1$ and $\alpha_2$. Further, as we show in Proposition~\ref{prop: existence and uniqueness of the x,y}, for $(\alpha_1, \alpha_2)\in\mathcal{R}_1$, there exists a unique such solution $x_{\star}=x_{\star}(\alpha_1, \alpha_2)$ in the interval $[\frac{1}{2},1]$, and that setting $y_{\star}=\sqrt{\alpha_1 -(x_{\star}) ^2 }$ we have $x_{\star}\leq \sqrt{\alpha_1}$ and $x_{\star}+y_{\star}\leq 1$, yielding the canonical pair $(x_{\star}, y_{\star})$. The boundary between $\mathcal{R}'_1$ and  $\mathcal{R}_1\setminus \mathcal{R}'_1$ then corresponds to the solutions $(\alpha_1, \alpha_2)\in \mathcal{R}_1$ to the equation
	\[2(x_{\star}(\alpha_1, \alpha_2))^2+(1-x_{\star}(\alpha_1, \alpha_2)-\sqrt{\alpha_1 -(x_{\star}(\alpha_1, \alpha_2))^2})^2=1.\]
\end{remark}
\begin{definition}
Let $\mathcal{R}_2$ denote the collection of $(\alpha_1, \alpha_2)\in[0,1]^2$ satisfying
\begin{align*}
\alpha_1 \geq \max\left(2-2\sqrt{\alpha_2}, 1-2\sqrt{\alpha_2}+2\alpha_2\right).
\end{align*}	
\end{definition} 
\noindent Note that for all pairs $(\alpha_1, \alpha_2)\in \mathcal{R}_1\cup \mathcal{R}_2$ we have $\frac{1}{4}\leq \alpha_2\leq \alpha_1$ and $\frac{1}{2}< \alpha_1$.
See Figure~\ref{figure:Rregions} for a picture of the regions $\mathcal{R}'_1$ and $\mathcal{R}_2$.
\begin{figure}\centering
	\includegraphics[scale=0.8]{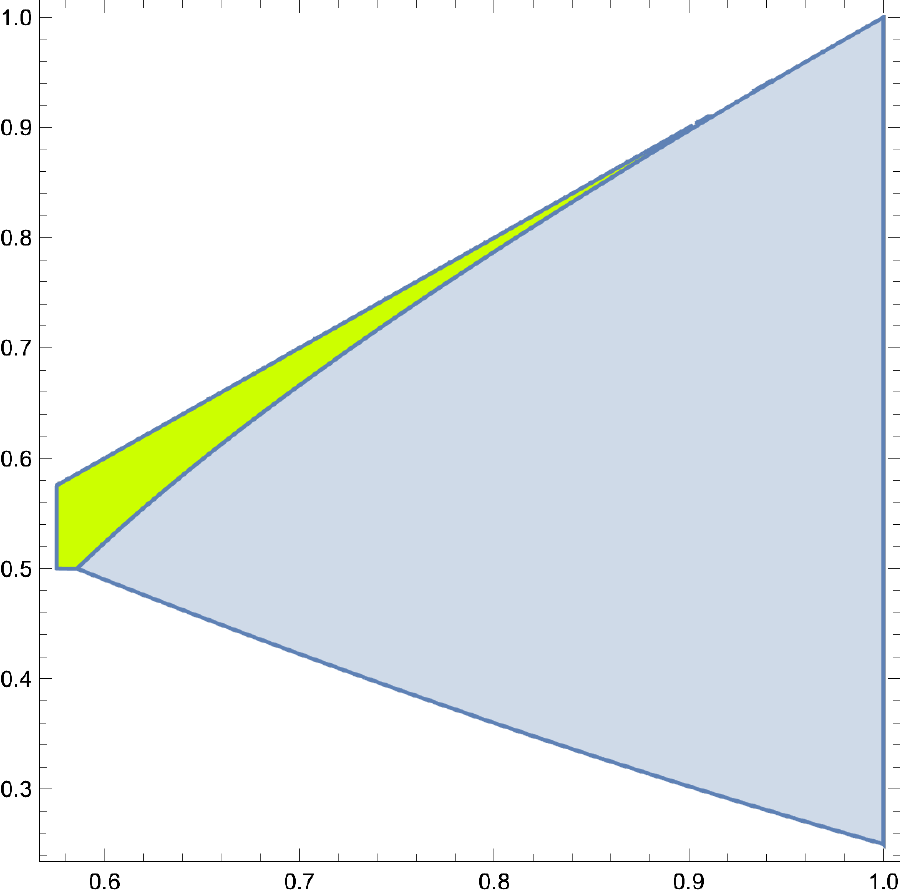}
	\caption{The regions $\mathcal{R}'_1$ (upper part, in green) and $\mathcal{R}_2$ (lower part, in blue) in the $(\alpha_1, \alpha_2)$ plane.}\label{figure:Rregions}
\end{figure}

Before stating our main result, we record a useful observation of Aharoni et al~\cite{AharoniDeVosdelaMazaMontejanoSamalin20}. Suppose that there exists an $N$-vertex Gallai colouring template $\mathbf{G}$ with no rainbow triangle  satisfying $e(G_i) = \frac{\alpha_i}{2} N^2 + \varepsilon_i N^2$ for each $i\in \{1,2,3\}$, where the $\varepsilon_i$ are strictly positive real numbers. Write $\mathbf{G(k)}$ for the balanced blow-up of $\mathbf{G}$ obtained by replacing each vertex $v$ of $\mathbf{G}$ by a set of $k$ vertices $X_v$ and for each $i$ replacing each edge $uv\in E(G_i)$ by a complete balanced bipartite graph between $X_u$ and $X_v$. Then for any $C > 0$ and all $k$ sufficiently large, we have 
\[e(G(k)_i) =\frac{1}{2}\alpha_{i}(kN)^2 + \varepsilon_i (kN)^2 > \alpha_{i} \binom{kN}{2} + CkN.\]
Since $\mathbf{G}(k)$ is rainbow triangle-free, this implies the existence of Gallai colouring templates $\mathbf{J}$ on $n>N$ vertices with  $e(J_i) > \alpha_i \binom{n}{2} + Cn$ for each $i\in\{1,2,3\}$. In particular, it is enough to resolve Problem~\ref{problem: which densities are forcing} up to additive linear terms and with the normalisation factor $n^2$ replaced by the more conventional factor $\binom{n}{2}$.

With this observation in place, we can now state our main result: for any pair of densities $1\geq \alpha_1\geq \alpha_2\geq 0$, we determine the least $\alpha_3\leq \alpha_2$ such that $(\alpha_1, \alpha_2, \alpha_3)$ is a forcing triple.
\begin{theorem}\label{theorem: forcing densities}
There exists a constant $C>0$ such that for any $(\alpha_1, \alpha_2)\in [0,1]^2$ with $\alpha_1\geq \alpha_2$, the following hold.
\begin{enumerate}[(a)]
	\item If $(\alpha_1, \alpha_2) \in \mathcal{R}'_1$, then letting $(x,y)$ be its canonical representation and setting $\alpha_3:=1-x^2$, we have that: \begin{enumerate}[(i)]
		\item $\alpha_2 \geq \alpha_3$;
		\item for any $n\in \mathbb{N}$, if $\mathbf{G}$ is an $n$-vertex $3$-colouring template with $\vert E(G_i)\vert \geq \alpha_i\binom{n}{2}+Cn$ for all $i\in [3]$, then $\mathbf{G}$ contains a rainbow triangle;
		\item for any $n\in \mathbb{N}$, setting $a=\lfloor xn\rfloor$, $b=\lfloor yn \rfloor$ and $c=n-a-b$, the $n$-vertex $3$-colouring template $\mathbf{F}(a,b,c)$ satisfies $\vert E(F_i)\vert \geq \alpha_i\binom{n}{2}-Cn$ and contains no rainbow triangle.
	\end{enumerate}
	In particular, any triple $(\alpha_1', \alpha_2', \alpha_3')$ with $\alpha_i'<\alpha_i$ or $\alpha_i'=0$  for every $i\in \{1,2,3\}$ is not a forcing triple, while every triple $(\alpha_1', \alpha_2', \alpha_3')$ with $\alpha_i'>\alpha_i$ or $\alpha_i'=1$ for every $i\in \{1,2,3\}$ is a forcing triple.
	
	\item  If $(\alpha_1, \alpha_2) \in \mathcal{R}_2$, then setting $\alpha_3 :=2-\alpha_1-2\sqrt{\alpha_2}+\alpha_2$, we have that:
	\begin{enumerate}[(i)]
	\item $\alpha_2\geq \alpha_3$;
\item  for any $n\in \mathbb{N}$, if $\mathbf{G}$ is an $n$-vertex $3$-colouring template with $\vert E(G_i)\vert \geq \alpha_i\binom{n}{2}+Cn$ for all $i\in [3]$, then $\mathbf{G}$ contains a rainbow triangle;
	\item   for any $n\in \mathbb{N}$, setting $a=\lfloor \sqrt{\alpha_2}n \rfloor$, $b=\lfloor \frac{1-\alpha_1}{2\sqrt{\alpha_2}}n \rfloor$ and $c=n-a-b$, the $n$-vertex $3$-colouring template $\mathbf{H}(a,b,c)$ satisfies $\vert E(H_i)\vert \geq \alpha_i\binom{n}{2}-Cn$ and contains no rainbow triangle.
\end{enumerate}
	In particular, any triple $(\alpha_1', \alpha_2', \alpha_3')$ with $\alpha_i'<\alpha_i$ or $\alpha_i'=0$ for every $i\in \{1,2,3\}$ is not a forcing triple, while every triple $(\alpha_1', \alpha_2', \alpha_3')$ with $\alpha_i'>\alpha_i$ or $\alpha_i'=1$ for every $i\in \{1,2,3\}$ is a forcing triple.
\item If $(\alpha_1, \alpha_2)\notin \mathcal{R}'_1\cup \mathcal{R}_2$, then $(\alpha_1, \alpha_2, \alpha_2)$ is not a forcing triple.
\end{enumerate}
\end{theorem}

\begin{remark}
Note that $\mathcal{R}'_1$ and $\mathcal{R}_2$ meet along the curve $\alpha_1=1-2\sqrt{\alpha_2}+2\alpha_2$ from the point $(\frac{1}{2},2-\sqrt{2})$  to the point $(1,1)$ --- indeed, along this curve, it is easily checked that the canonical representation of $(\alpha_1, \alpha_2)$ is $(x,y)$ where $x=\sqrt{\alpha_1-(1-\sqrt{\alpha_2})^2} = \sqrt{\alpha_2}$ and $y=1-\sqrt{\alpha_2}$, and satisfies $2x^2+(1-x-y)^2=2\alpha_2\geq 1$. For $(\alpha_1, \alpha_2)$ along this curve, our extremal $3$-colouring templates $\mathbf{H}$ and $\mathbf{F}$ both have $\vert C\vert =o(n)$ and (up to changing at most $o(n^2)$ edges into non-edges and vice versa in each of the colour classes) degenerate down to the same $3$-colouring template $\mathbf{G}$ on $A\sqcup B=[n]$ with $\vert A\vert =\lfloor\sqrt{\alpha_2}n\rfloor$, $\vert B\vert =n-\vert A\vert$ and colour classes $G_1= A^{(2)}\cup B^{(2)}$, $G_2=A^{(2)}$ and $G_3=[n]^{(2)}\setminus A^{(2)}$.
\end{remark}
\noindent As a consequence of Theorem~\ref{theorem: forcing densities}, we settle Conjecture~\ref{conjecture: product extremal}:
\begin{corollary}\label{cor: conjecture true}
	Conjecture~\ref{conjecture: product extremal} is true.
\end{corollary}
\noindent 

\subsection{Further remarks and open problems}

\noindent \textbf{Minimum degree conditions:} in both of our  extremal colouring templates $\mathbf{F}$ and $\mathbf{H}$, there are colour classes with isolated vertices.  Indeed, we have $\delta(F_1) = \delta(F_2) = 0$ (by considering vertices in $C$ and $B$ respectively) and $\delta(G_2)=0$ (by considering vertices in $B\cup C$). Given this, it is natural to ask how Problem~\ref{problem: which densities are forcing} changes when we impose minim-degree rather than density conditions.

We study this question in a companion paper~\cite{FRMarkstromRaty22}, in which given $\delta(G_1)$ we determine the maximum possible value of $\delta(G_2)+\delta(G_3)$ in a Gallai colouring template $\mathbf{G}$. It turns out the extremal behaviour for this problem is starkly different from the one we established for Problem~\ref{problem: which densities are forcing} in this paper. Indeed, the maximum possible value of $\delta(G_2) + \delta(G_3)$ jumps from $\frac{2n}{r}$ to $\frac{2n}{r+1}$ when $\delta_1(G)$ increases from $n-\lceil\frac{n}{r}\rceil$ to $\lceil n-\frac{n}{r}\rceil+1$, in contrast to the more continuous behaviour seen in Theorem~\ref{theorem: forcing densities}.

\noindent \textbf{Other cliques:} in~\cite{AharoniDeVosdelaMazaMontejanoSamalin20}, Aharoni, DeVos, de la Maza, Montejano and \v{S}\'amal asked what happens when the triangle $K_3$ is replaced with a complete graph $K_r$ on $r$ vertices when $r \geq 4$. 
\begin{question}
	Let $r \geq 4$. What is the smallest real number $\delta_r$ so that for all $n$ sufficiently large, any $n$-vertex ${r \choose 2}$-colouring template $\mathbf{G}$ with $\min\left\{\vert E(G_1) \vert, \dots, \vert E(G_{{r \choose 2}})\vert  \right\} > \delta_r \frac{n^2}{2}$ must contain a rainbow copy of $K_{r}$?
\end{question}
\noindent By considering $G_1=G_2=\cdots = G_{\binom{r}{2}}=T_{r-1}(n)$, the $(r-1)$-partite Tur\'an graph, it is clear $\delta_r \geq  1 - \frac{1}{r-1}$. Is this bound tight for any $r$?

\noindent\textbf{Other graphs:} besides larger cliques, one can ask for conditions guaranteeing the existence of rainbow copies of some other graph $H$. Babi\'nski and Grzesik~\cite{BabinskiGrzesik22} recently considered this problem when $H = P_3$, the path on $4$ vertices with  $3$ edges. For every $r\geq 3$, they determined the value of the least $\alpha(r, P_3)\geq 0$ such that for all $\alpha>\alpha(r, P_3)$ and all $n$ sufficiently large, every $n$-vertex $r$-colouring template $\mathbf{G}$ with $\min\left(\vert E(G_1) \vert, \cdots, \vert E(G_r) \vert \right) \geq \alpha n^2$ must contain a rainbow $P_3$.

In a similar direction, Frankl, Gy\H{o}ri, He, Lv, Salia, Tompkins, Varga and Zhu~\cite{FranklGyoriHeLvSaliaTompkins22} successfully determined the (asymptotic behaviour of the)  maximum of the geometric mean of the colour classes in $r$-colouring templates with no rainbow copy of $H$ when $r\in\{3,4\}$ and $H=P_3$ and when $r=4$ and $H=P_4$, the path on five vertices. It would be interesting to obtain generalisation of both of these results for longer paths.

\noindent\textbf{Stability, colourings vs templates:} we expect that the proof of Theorem ~\ref{theorem: forcing densities} can be adapted to give stability versions of our results, but we had not explored this further due to the length of the paper. Finally, we focused in this work on colouring \emph{templates}, in which colour classes may overlap. Following Erd{\H o}s and Tuza~\cite{ErdosTuza93}, one could instead consider analogous problems for colourings of $K_n$ or of subgraphs of $K_n$. Can one obtain analogues of Theorem~\ref{theorem: forcing densities} in this setting?

\subsection{Notation}
As noted above, we write $[n]:=\{1,2,\ldots n\}$,  $S^{(2)}:=\{\{s,s'\}:\  s,s' \in S, s\neq s'\}$ and $(S,T)^{(2)}:=\{\{s,t\}: \ s\in S, t\in T\}$. Where convenient, we identify $G_i$ with its edge-set  $E(G_i)$.  We also write $xy$ for $\{x,y\}$. We use $G_i[X]$ and $G_i[X,Y]$ as a notation for the subgraph of $G_i$ induced by the vertex-set $X$ and for the bipartite subgraph of $G_i$ induced by the bipartition $X\sqcup Y$ respectively. Throughout the remainder of the paper, we shall use $\vert G_i\vert$, $\vert G_i[X]\vert$ and $\vert G_i[X,Y]\vert$  as shorthands for $\vert E(G_i)\vert$, $\vert E(G_i[X])\vert$ and $\vert E(G_i[X,Y])\vert$ respectively. We use Landau big O notation, and note that $g=O(f)$ or $g=o(f)$ is an assertion about the order of $g$ and not its sign (so we do not differentiate between $1-o(1)$ and $1+o(1)$, for example).

Given a $3$-colouring template $\mathbf{G}$ on a set $V$, we call a pair $xy\in V^{(2)}$ a \emph{rainbow edge} if $xy\in \bigcap_{i=1}^3G_i$. Further, we call a pair $xy$ which is contained in at least two of the colour classes $G_1, G_2, G_3$ a \emph{bi-chromatic edge}. The following notion of density for a colouring template will be a useful tool in our analysis:
\begin{definition}[Colour density vector]
	Given an $r$-colouring template $\mathbf{G}=(G_1, G_2, \ldots, G_r)$ on an $n$-set $V$, the colour density vector of $\mathbf{G}$ is 
	\[\mathbf{\rho}(\mathbf{G}):=\left(\frac{\vert G_1\vert}{\binom{n}{2}},\  \frac{\vert G_2\vert}{\binom{n}{2}},\  \ldots , \ \frac{\vert G_r\vert}{\binom{n}{2}}\right).\]

\end{definition}

\section{Critical colour densities for rainbow triangles}\label{section: edge density}
\subsection{Preliminary remarks}\label{subsection: preliminaries}
We begin by analysing the colour density vectors yielded by Constructions~\ref{construction: F(a,b,c)} and~\ref{construction: H(a,b,c)}.
\begin{proposition}\label{prop: construction analysis}
For $a=xn$, $b=yn$ and $c=zn$, the  colour density vectors of $\mathbf{F}$ and $\mathbf{H}$ are
\[ \left(x^2+y^2, \ x^2+z^2,\  1-x^2 \right)  + \left(O(n^{-1}),\ O(n^{-1}),\ O(n^{-1})\right)\]	
and
\[\left(1-2xy,\  x^2,\ (1-x)^2 +2xy\right)  + \left(O(n^{-1}),\ O(n^{-1}),\ O(n^{-1})\right) \]
respectively. In particular, for $z=0$ (and thus $x+y=1$) they coincide asymptotically and are both equal to
$ \left(x^2+(1-x)^2,\ x^2,\ 1-x^2\right)  + \left(O(n^{-1}),\ O(n^{-1}),\ O(n^{-1})\right)$.	
\end{proposition}
\begin{proof}
	Simple calculation.
\end{proof}
\noindent Recall that  $\tau=\frac{4-\sqrt{7}}{9}$. The next two propositions establish that certain $(\alpha_1, \alpha_2, \alpha_3)$ are trivially not forcing triples and that for $(\alpha_1, \alpha_2)$ there exists a unique canonical representation $\alpha_1=x^2+y^2$, $\alpha_2=x^2+z^2$ with $x\geq 1/2$, $0\leq y \leq 1-x$ and $x+y+z=1$.
\begin{proposition}\label{prop: trivial conditions}
Let $(\alpha_1, \alpha_2, \alpha_3)$ be a triple of elements of $[0,1]$ with $\alpha_1\geq \alpha_2\geq \alpha_3$. 	If any of the following hold, then $(\alpha_1,\alpha_2, \alpha_3)$ is not a forcing triple:
	\begin{enumerate}[(a)]
		\item  $\alpha_1< \frac{1+\tau^2}{2}=\frac{52-4\sqrt{7}}{81}$;
\item $\alpha_2<\frac{1}{4}$;
\item $\alpha_1+\alpha_2<1$;
\item $\alpha_1=x^2+y^2$ and $\alpha_2=x^2+(1-x-y)^2$ for some non-negative reals $x,y$ with $x+y\leq 1$ and $2x^2+(1-x-y)^2<1$.
	\end{enumerate}
\end{proposition}
\begin{proof}
For each of the four cases (a)--(d), we construct a suitable $n$-vertex Gallai $3$-colouring template based on $\mathbf{F}=\mathbf{F}(a,b,c)$ whose colour density vector is coordinate-wise asymptotically strictly greater than $(\alpha_1, \alpha_2, \alpha_3)$ (possibly after rearranging the order of the colours). 
Since $\mathbf{F}$ is rainbow $K_3$-free, 
this suffices to show that $(\alpha_1, \alpha_2, \alpha_3)$ is not a forcing triple.

\noindent \textbf{Case (a): $\alpha_1< \frac{1+\tau^2}{2}$.} Set $a=n-2\lceil\tau n\rceil$, $b=c=\lceil \tau n\rceil$. Then $\mathbf{F}(a, b ,c)$ has asymptotic colour density vector $\left(\frac{1+\tau^2}{2}\right)\cdot(1,1,1)$. For $\varepsilon >0$ chosen sufficient small, this is pointwise strictly greater than $(\alpha_1, \alpha_2, \alpha_3)+\varepsilon\cdot (1,1,1)$. Thus $(\alpha_1,\alpha_2, \alpha_3)$ is not a forcing triple.

\noindent \textbf{Case (b): $\alpha_2<\frac{1}{4}$.} Set $a=0$, $b=\lceil \frac{n}{2}\rceil$, $c=n-b$. Then $\mathbf{F}(a, b, c)$ has asymptotic colour density vector $(\frac{1}{4},\frac{1}{4},1)$. For $\varepsilon >0$ chosen sufficiently small, this is pointwise strictly greater than $(\alpha_2, \alpha_3, \alpha_1)+\varepsilon\cdot(1,1,0)$ (since $\alpha_3\leq\alpha_2$). Rearranging colours, it immediately follows that $(\alpha_1,\alpha_2, \alpha_3)$ is not a forcing triple.

\noindent \textbf{Case (c): $\alpha_1+\alpha_2<1$.}  Pick $\varepsilon>0$ sufficiently small so that $1-\alpha_2-4\varepsilon>\alpha_1$. Set $a=\lceil n\sqrt{\alpha_2 +2\varepsilon}\rceil$, $b=n-a$, $c=0$. Then $\mathbf{F}(a, b, c)$ has asymptotic colour density vector $(\alpha_2+2\varepsilon+(1-\sqrt{\alpha_2+2\varepsilon})^2,  \alpha_2+2\varepsilon, 1-\alpha_2-2\varepsilon)$, which is strictly greater than $(\alpha_2, \alpha_3, \alpha_1)+\varepsilon\cdot (1,1,1)$. Rearranging colours, it immediately follows that $(\alpha_1,\alpha_2, \alpha_3)$ is not a forcing triple.

\noindent \textbf{Case (d)}: $\alpha_1=x^2+y^2$, $\alpha_2=x^2+(1-x-y)^2$  and $\alpha_2+x^2<1$. 
Observe that $2x^2<1$, whence $x<1/\sqrt{2}$.  Since $\alpha_1\geq \alpha_2$, this implies $y\geq (1-x)/2>0$. Further, by Case (c) above, we may assume $1\leq \alpha_1+\alpha_2$; since $\alpha_1+\alpha_2 \leq 2x^2+(1-x)^2$, we deduce from this that $x\geq 2/3$ and in particular $x>y$.

Pick $\varepsilon$: $0<\varepsilon<y$ sufficiently small so that $\alpha_2+\varepsilon^2<1-(x+\varepsilon)^2$. Then for $a=\lfloor (x+\varepsilon)n \rfloor$, $b=\lfloor (y-\varepsilon)n\rfloor$ and $c= n-a-b$, the $3$-colouring template $\mathbf{F}(a,b,c)$ contains no rainbow triangles and has asymptotic colour density vector $((x+\varepsilon)^2+ (y-\varepsilon)^2, (x+\varepsilon)^2+(1-x-y)^2, 1- (x+\varepsilon)^2)$, which is pointwise strictly greater than $(\alpha_1, \alpha_2, \alpha_2)+ \varepsilon^2 \cdot (2,1,1)$ (here in the first coordinate we used the fact that $x> y$). Since $\alpha_2\geq \alpha_3$, it immediately follows that $(\alpha_1,\alpha_2, \alpha_3)$ is not a forcing triple.
\end{proof}

\begin{proposition}\label{prop: existence and uniqueness of the x,y}
	Given non-negative real numbers $\alpha_1, \alpha_2$	satisfying $\alpha_1\geq \frac{1}{2}$ and
	 $\frac{\alpha_1+\sqrt{2\alpha_1-1}}{2}\leq \alpha_2\leq \alpha_1$, there exist a unique triple $(x,y, z)\in [0,1]^3$ with $x+y+z=1$ and $x\geq \frac{1}{2}$ such that
	\begin{align}\label{eq: unique expression for alpha1, alpha2}
	\alpha_1= x^2+y^2 && \alpha_2 = x^2+z^2. 
	\end{align}
\end{proposition}
\begin{proof}	
Set $y(x):=\sqrt{\alpha_1-x^2}$	and $z(x):=1-x-y(x)$.  Our goal is to show there exists a unique solution $x_{\star}$ to $x^2+(z(x))^2=\alpha_2$ with $x\geq\frac{1}{2}$, $y(x)$ real and $z(x)\geq 0$.

Solving the appropriate quadratic equations, it is easily checked that for $x\in [\frac{1}{2},1]$ we have $z(x)\geq 0$ for $x\geq x_0= \frac{1+\sqrt{2\alpha_1-1}}{2}$ and $y(x)\geq z(x)$ for $x\leq x_1=\frac{1+2\sqrt{5\alpha_1-1}}{5}$. It is clear geometrically that $x_0\leq x_1$ (these values of $x$ corresponding as they do to intersections of the circle $x^2+y^2=\alpha_1$ with the lines $y=1-x$ and $y=(1-x)/2$ in the first quadrant of the plane). Further, solving another two quadratic equations, it is easily checked that  $x_1\leq \sqrt{\alpha_1}$ with equality if and only if $\alpha_1=1$, so that $y(x)$ is real in the interval $[x_0, x_1]$.

Now, $(x_0)^2+(z(x_0))^2=(x_0)^2=\frac{\alpha_1+\sqrt{2\alpha_1-1}}{2}\leq \alpha_2$ and $(x_1)^2+(z(x_1))^2=(x_1)^2+(y(x_1))^2=\alpha_1\geq \alpha_2$. The existence of an $x_{\star}\in [x_0, x_1]$ such that $(x_{\star})^2 +(z(x_{\star}))^2=\alpha_2$ thus follows from the intermediate value theorem.

It remains to show the uniqueness of this solution. Suppose there exists $x=x_{\star}+d_x$ for some $d_x\geq 0$ and $y, z$ with $x+y+z=1$ such that $(x,y,z)$ satisfies~\eqref{eq: unique expression for alpha1, alpha2}. Clearly we must have $y = y(x_\star)-d_y$ and $z = z(x_{\star})-d_z$ for some non-negative $d_y, d_z$ with $d_x=d_y+d_z$ (otherwise one of the equations in~\eqref{eq: unique expression for alpha1, alpha2} or the condition $x+y+z=1$ must fail). Since $x\geq \frac{1}{2}$ we have $y\leq \frac{1}{2}$ and $(x_{\star}+d_x)^2+(y(x_{\star})-d_y)^2\geq (x_{\star})^2+(y(x_{\star}))^2 +(d_x-d_y)$. In particular, $d_y\geq d_x$. Then $d_x=d_y+d_z$ implies $d_z=0$, which in turn implies $d_x=0$ (else $x^2+z^2>(x_{\star})^2 +(z(x_{\star}))^2$) and hence $d_y=0$, and the uniqueness of our triple $(x_{\star},y(x_{\star}),z(x_{\star}))$.
\end{proof}
\begin{definition}[Good pair]
	We say that a pair of non-negative real numbers $(\alpha_1,\alpha_2)$ from $[0,1]^2$ is a \emph{good pair} if 
	\begin{align*}
	\max\left\{\frac{1}{4}, \frac{\alpha_1+\sqrt{2\alpha_1-1}}{2}\right\} \leq \alpha_2 &&  \max\left\{\alpha_2, 1-\alpha_2, \frac{1+\tau^2}{2}  \right\} \leq \alpha_1,
	\end{align*}
	and in addition the unique $(x,y,z)\in [0,1]^3$ with $x+y+z=1$ and $x\geq \frac{1}{2}$ such that~\eqref{eq: unique expression for alpha1, alpha2} holds satisfies $2x^2+z^2\geq 1$. Given a good pair, we refer to this unique $(x,y,z)$ (whose existence is guaranteed by Proposition~\ref{prop: existence and uniqueness of the x,y}) as the \emph{canonical representation} of $(\alpha_1, \alpha_2)$. 
\end{definition}

\subsection{Proof strategy}\label{subsection: proof strategy}
We divide the proof of Theorem~\ref{theorem: forcing densities} into two parts, depending on whether or not the edge densities $\alpha_1$ and $\alpha_2$ of the two largest colour classes satisfy $\alpha_1 \leq \alpha_2+ (1-\sqrt{\alpha_2})^2$. 
 In both cases, we prove a technical statement of the form `if the colour classes of a colouring template satisfy certain inequalities, then it must contain a rainbow triangle'. To do so, we consider a putative minimal counterexample $\mathbf{G}$ to our technical statement, and use its minimality to rule out the existence of rainbow edges. 

We then consider a largest matching $M$ of bi-chromatic  edges in $\mathbf{G}$, which we use to obtain a partition of $V=V(\mathbf{G})$ into sets $V_{ij}$ of vertices meeting a bi-chromatic edge of $M$ in colours $ij$ and a left-over set $D$. We perform a series of modification of $\mathbf{G}$ to obtain a new colouring template $\mathbf{G''}$ such that the sizes of the colour classes of $\mathbf{G''}$ satisfy the same inequalities as those of $\mathbf{G}$ up to some small $O(n)$ error terms. The crux is, however, that $\mathbf{G''}$ is very well-structured with respect to the partition obtained in the previous step, so that we have a good control over the sizes of its colour classes. In the final step of the argument, we use this information to derive a contradiction from our family of inequalities.

The idea of considering a largest matching of bi-chromatic edges and modifying $\mathbf{G}$ based on the resulting partition appeared previously in the work of Aharoni, DeVos, de la Maza, Montejanos and \v{S}\'amal~\cite{AharoniDeVosdelaMazaMontejanoSamalin20}, more specifically their key Lemma 2.3 which inspired our approach in the case $\alpha_1 \leq \alpha_2+ (1-\sqrt{\alpha_2})^2$.

 An important additional ingredient in our proof in the case $\alpha_1 > \alpha_2+ (1-\sqrt{\alpha_2})^2$ is the idea of looking a vertex-minimal counterexample $\mathbf{G}$ which also maximises the size of the largest colour class $G_1$. Indeed, this allows us to `push' $\mathbf{G}$ towards a much more amenable bipartite extremal structure, which we are able to analyse.
\subsection{The $\mathbf{F}$-extremal region: the case $\alpha_1\leq 
\alpha_2 +(1-\sqrt{\alpha_2})^2$}\label{section: easy case}
Note that for $\alpha_1\geq 1/2$ and $\alpha_2\geq 1/4$, the inequality for $\alpha_1, \alpha_2$ we have in this case is equivalent to the lower bound for $\alpha_2$ we had in our definition of a good pair in Section~\ref{subsection: preliminaries}:
\begin{align}\label{eq: easy case condition}\alpha_2\leq \alpha_1\leq \alpha_2 +(1-\sqrt{\alpha_2})^2 && \Leftrightarrow && \frac{\alpha_1+\sqrt{2\alpha_1-1}}{2}\leq \alpha_2\leq \alpha_1.\end{align}
\begin{theorem}\label{theorem: technical version, easy case}
Let $(\alpha_1, \alpha_2)$ be a good pair and let $(x,y,z)$ be its associated canonical representation. Set $\alpha_3:= 1-x^2$. If $\mathbf{G}$ is a $3$-colouring template on $n$ vertices satisfying
\begin{align}\label{eq: bound on Gi Gj}
\vert E(G_i)\vert +\vert E(G_j)\vert \geq \left(\alpha_i+\alpha_j\right)\binom{n}{2}+5n
\end{align}
for all distinct $i, j \in [3]$,  then $\mathbf{G}$ contains a rainbow triangle.
\end{theorem}

\begin{proof}
Observe that for $n\leq 6$, the statement of Theorem~\ref{theorem: technical version, easy case} is vacuous, since $5n\geq 2\binom{n}{2}$. Suppose Theorem~\ref{theorem: technical version, easy case} is false, and let $N\geq 7$ be the least value of $n$ for which there exists a Gallai $3$-colouring template $\mathbf{G}$ which provides a counterexample. Without loss of generality, we may assume the vertex-set of $\mathbf{G}$ is $V=[N]$. We begin our proof with an analogue of~\cite[Lemma 2.4]{AharoniDeVosdelaMazaMontejanoSamalin20}, which establishes inter alia that there are no rainbow edges.
\begin{lemma}\label{lemma: case 1, no triple matching}
For every non-empty proper subset $X$ of $V$, at least one of the induced subgraphs $G_i[X]$, $i\in [3]$, fails to contain a perfect matching. 
\end{lemma}
\begin{proof}
Let $X$ be a $2\ell$-set in $V$ with $0<\ell<N/2$. Suppose for a contradiction that the graphs $G_1[X]$,  $G_2[X]$ and $G_3[X]$ contain perfect matchings $M_1$, $M_2$ and $M_3$ respectively. We shall bound $\vert G_i[V\setminus X]\vert +\vert G_j[V\setminus X]$ for all distinct colour pairs $ij\in [3]^{(2)}$.

Fix a colour $k\in [3]$, and let $i,j$ denote the other two colours in $[3]$. Let $vv'$ be an edge of $M_k$. Then every vertex $u \in V\setminus X$ can send at most $2$ edges in colour $i$ or $j$ to $\{x,x'\}$ (for otherwise we have a rainbow triangle).  Summing over all edges of $M_3$, it follows that 
\begin{align}\label{eq: bound on ei(X,V-X)+ ej(X, v-X)}
\vert G_i[X, V\setminus X]\vert + \vert G_j[X, V\setminus X]\vert \leq \sum_{vv'\in M_k} \left(\vert G_i[\{v,v'\}, V\setminus X]\vert +\vert G_j[\{v,v'\}, V\setminus X]\vert\right)&\leq 2\ell(N-2\ell).
\end{align}
Next  we show that $\vert G_i[X]\vert + \vert G_j[X]\vert \leq 2\ell^2$. If $\ell=1$, we have nothing to show since $2\binom{2l}{2}=2\ell^2$. On the other hand if $\ell\geq 2$, then consider an edge $uu' \in (G_i\cap G_j)[X]\setminus M_k$. Since $M_k$ is a perfect matching and $uu'\notin M_k$, there exist distinct $v,v' \in X\setminus\{u,u'\}$ such that $uv, u'v'\in M_k$. This in turn implies that $uv', u'v \notin (G_i\cup G_j)[X]\setminus M_k$ (since otherwise one of the sets $\{u,u',v'\}$, $\{u,u',v\}$ would contain a rainbow triangle). Note that the vertices $v,v'$ are uniquely specified by $uu'$ and the matching $M_k$.

Thus given any $uu'\in X^{(2)}\setminus M_k$ we can define a pair of edges $F(uu')=\{uv', u'v\}$, with $v,v'$ as above, such that either $uu'\notin G_i\cap G_j$ or $uv', u'v \notin G_i\cup G_j$. Observe that $F(uu')\cap F(ww')\neq \emptyset$ if and only if $M_k$ contains a matching from $\{u,u'\}$ to $\{w,w'\}$ (i.e.\  if and only if $ww'=vv'$), in which case $F(uu')=F(ww')$. In particular we have that $\vert (G_i\cap G_j)[X]\setminus M_k \vert \leq \vert X^{(2)}\setminus (G_i\cup G_j \cup M_k)\vert$ and thus
\begin{align}\label{eq: bound on ei(X)+ ej(X)}
\vert G_i[X]\vert +\vert G_j[X]\vert &\leq 2\vert M_k\vert + \vert (G_i\cup G_j)[X]\setminus M_k\vert +\vert (G_i\cap G_j)[X]\setminus M_k\vert &\leq 2\ell +\left(\binom{2\ell}{2}-\ell\right)=2\ell^2.
\end{align}
Putting~\eqref{eq: bound on ei(X,V-X)+ ej(X, v-X)} and~\eqref{eq: bound on ei(X)+ ej(X)} together, we have
\begin{align}\label{eq: bound on Gi Gj -X}
\vert G_i[V\setminus X]\vert +\vert G_j[V\setminus X]\vert &= \vert G_i\vert   -\vert G_i[X, V \setminus X]\vert -\vert G_i[X]\vert + \vert G_j\vert   -\vert G_j[X, V \setminus X]\vert -\vert G_j[X]\vert \notag \\
&\geq \left(\alpha_i +\alpha_j\right)\binom{N}{2} +5N-2\ell (N-2\ell) -2\ell^2.
\end{align}
\begin{claim}\label{claim: removing X works}
	$\left(\alpha_i +\alpha_j\right)\binom{N}{2} +5N-2\ell (N-2\ell) -2\ell^2\geq \left(\alpha_i +\alpha_j\right)\binom{N-2\ell}{2} +5(N-2\ell)$.
\end{claim}	
\begin{proof}
Rearranging terms, what we must show is
\begin{align}
\left(\alpha_{i}+\alpha_{j}-1\right)2\ell \left(N-\ell\right)+\left(10-\alpha_{i}-\alpha_{j}\right)\ell\geq 0.\label{eq: removing X works, rearranged bound}
\end{align}	
Note first of all that $\alpha_i+\alpha_j\geq 1$. 	Indeed, since $(\alpha_1, \alpha_2)$ is a good pair, this is by definition the case for $\{i,j\}=\{1,2\}$. Further, the definition of $\alpha_3:=1-x^2$ ensures $\alpha_1+\alpha_3=1+y^2$ and $\alpha_2+\alpha_3=1+z^2$ are both at least $1$.  Now, since $X$ was a proper non-empty subset of $V$, we have $N>2\ell$, and hence the first term in the sum on the left hand-side of~\eqref{eq: removing X works, rearranged bound} is non-negative. As $\alpha_i+\alpha_j\leq 2$, the second term in~\ref{eq: removing X works, rearranged bound} is strictly positive.  Thus~\eqref{eq: removing X works, rearranged bound} holds, as required.	
\end{proof}	
\noindent Since $i,j$ were arbitrary, it follows from~\eqref{eq: bound on Gi Gj -X} and Claim~\ref{claim: removing X works} that $\mathbf{G}[V\setminus X]$ is a Gallai $3$-colouring template on $n=N-\vert X\vert < N$ vertices satisfying~\eqref{eq: bound on Gi Gj}, and hence a smaller counterexample to Theorem~\ref{theorem: technical version, easy case}, contradicting the minimality of $N$. 	
\end{proof}
\noindent As in~\cite{AharoniDeVosdelaMazaMontejanoSamalin20}, we have the following corollary to Lemma~\ref{lemma: case 1, no triple matching}:
\begin{proposition}[Observation 2.5 in~\cite{AharoniDeVosdelaMazaMontejanoSamalin20}]\label{prop: obs2.5}
Let $xx'$ and $yy'$ be vertex-disjoint pairs from $V^{(2)}$. Let $\{i,j,k\}=[3]$. Then the following hold:
\begin{enumerate}[1.]
	\item if $xx', yy' \in G_i\cap G_j$, then either\\
	 (a) $\vert G_k[xx', yy']\vert=0$, or \\(b) $\vert G_k[xx', yy']\vert=1$ and $\vert G_i[xx', yy']\vert , \vert G_j[xx', yy']\vert\leq 2$, or\\
	  (c) $\vert G_k[xx', yy']\vert=2$ and $\vert G_i[xx', yy']\vert =\vert G_j[xx', yy']\vert=0$;
	\item   if $xx'\in G_i\cap G_j$ and $yy'\in G_i\cap G_k$, then either\\
	 (a) $\sum_{i=1}^3 \vert G_i[xx, yy']\vert \leq 4$, or\\
	  (b) $\vert G_i[xx', yy']\vert=3$, $\vert G_j[xx', yy']\vert =\vert G_j[xx', yy']\vert =1$, this latter possibility occurring if and only if we have (up to permutations of the pairs $jk$, $xx'$ and $yy'$) $xy\in G_i\cap G_k$, $x'y'\in G_i\cap G_j$ and $xy'\in G_i$.
\end{enumerate}
\end{proposition}
\begin{proof}
	Identical to the (simple case analysis in the) proof of~\cite[Observation~2.5]{AharoniDeVosdelaMazaMontejanoSamalin20} but with our Lemma~\ref{lemma: case 1, no triple matching} replacing~\cite[Lemma 2.4]{AharoniDeVosdelaMazaMontejanoSamalin20}.
\end{proof}
Still following Aharoni et al's approach from~\cite{AharoniDeVosdelaMazaMontejanoSamalin20}, we consider a largest matching $M$ of bi-chromatic edges (called \emph{digons} in~\cite{AharoniDeVosdelaMazaMontejanoSamalin20}), to obtain a partition of the vertex set. For $ij\in [3]^{(2)}$, set $M_{ij}:=M\cap G_i\cap G_j$, and let $V_{ij}$ denote the collection of vertices contained in an edge of $M_{ij}$. Set $D:=V\setminus \left(V_{13}\sqcup V_{23}\sqcup V_{23}\right) $ to be the set of vertices not contained in an edge of $M$. As observed by Aharoni et al, one can perform some local modifications of $\mathbf{G}$ to obtain a new colouring template $\mathbf{G}''$ which is well-structured with respect to the partition $V=V_{13}\sqcup V_{12}\sqcup V_{23}\sqcup D$, may possibly contain rainbow triangles, but importantly satisfies the bounds~\eqref{eq: bound on Gi Gj} up to a small correction term which is linear in $N$. More explicitly, combining~\cite[Claims 1--3]{AharoniDeVosdelaMazaMontejanoSamalin20}, one obtains the following:
\begin{proposition}[Claims 1--3 in~\cite{AharoniDeVosdelaMazaMontejanoSamalin20}]\label{prop: easy case, good structure}
There exists a $3$-colouring template $\mathbf{G''}$  on $V$ 
such that the following hold:
\begin{enumerate}[(i)]
	\item the bound
	$\vert G''_{i}\vert +\vert G''_{j}\vert \geq \vert G_{i}\vert +\vert G_{j}\vert-\frac{3}{2}N > \left(\alpha_{i}+\alpha_{j}\right)\binom{N}{2}+2N$
	holds for all distinct $i$ and $j$; 
	\item $\bigcap_{i=1}^3 G''_i=\emptyset$ (i.e.\ there are no rainbow edges)
	\item for all $ij\in [3]^{(2)}$, $(G''_i\cap G_j'')[V_{ij}]={(V_{ij})}^{(2)}$  (i.e.\ $V_{ij}$ induces a bi-chromatic clique of edges in colours $i$ and $j$, and thus by condition (ii) above contains no edge in the third colour);
	\item  there are no bi-chromatic edges inside $D$ or between distinct sets $V_{ij}$, $ij\in[3]^{(2)}$;
	\item if $y\in D$ and $xx'$ is an edge in $M_{ij}=M\cap{(V_{ij})}^{(2)}$, then 
	$\vert G''_1[\{y\},\{x,x'\}]\vert + \vert G''_2[\{y\}, \{x,x'\}]\vert +\vert G''_3[\{y\},\{x,x'\}]\vert\leq3$, 
	with equality if and only if $\vert G''_i[\{y\}, \{x,x'\}]\vert + \vert G''_j[\{y\}, \{x,x'\}]\vert =3$. 
\end{enumerate}
\end{proposition}
\begin{proof}
Immediate from the construction of the modified colour classes $G''_i$, $i\in [3]$ in ~\cite[Claims 1--3]{AharoniDeVosdelaMazaMontejanoSamalin20} (which only rely on Lemma~\ref{lemma: case 1, no triple matching}, Proposition~\ref{prop: obs2.5} and the self-contained graph theoretic lemma~\cite[Lemma 2.2]{AharoniDeVosdelaMazaMontejanoSamalin20}). 
Note that we started out with a slightly larger linear term in our inequality~\eqref{eq: bound on Gi Gj}, whence the slightly larger term in the expression to the right of  the last inequality in condition (i).
\end{proof}
Set $a_{ij}:=\vert V_{ij}\vert/N$ and $d:=\vert D\vert /N$. We are now ready to proceed with the last part of the proof of Theorem~\ref{theorem: technical version, easy case}, where we use the structure of the colouring template $\mathbf{G}''$ to derive upper bounds for the sizes of its colour classes in terms of $(a_{12}, a_{13}, a_{23}, d)$ (Lemma~\ref{lemma: easy case, inequalities for the aij} below), which we then show contradict the lower bounds from Proposition~\ref{prop: easy case, good structure}(i) (Lemma~\ref{lemma: easy case, contradiction} below). Lemma~\ref{lemma: easy case, contradiction} is also the point in the proof of Theorem~\ref{theorem: technical version, easy case} where we depart from the approach of Aharoni et al~\cite{AharoniDeVosdelaMazaMontejanoSamalin20}.


\begin{lemma}\label{lemma: easy case, inequalities for the aij}
The following inequalities are satisfied:
\begin{align}
a_{12}(a_{12}+d)&> \alpha_1+\alpha_2-1=2x^2+y^2+z^2-1 \label{eq: a 12 bound},\\
a_{13}(a_{13}+d)&> \alpha_1+\alpha_3-1= y^2 \label{eq: a 13 bound},\\
a_{23}(a_{23}+d)&> \alpha_2+\alpha_3-1 =z^2 \label{eq: a 23 bound},\\
\sum_{ij}  a_{ij}(a_{ij} + d)& >  \alpha_1+\alpha_2+\alpha_3-1=x^2+y^2+z^2 \label{eq sum of a ij squared bound}\\
(a_{12})^2 + 2(a_{13})^2 + 2(a_{23})^2 + 2a_{13}d + 2a_{23}d& > 2\alpha_{1} + 2\alpha_{2} + 3\alpha_{3} -3= x^2 + 2y^2 + 2z^2,  \label{eq: a 12 second bound}\\
2(a_{12})^2 + (a_{13})^2 + 2(a_{23})^2 + 2a_{12}d + 2a_{23}d& > 2\alpha_{1} + 3\alpha_{2} + 2\alpha_{3}-3, \label{eq: a 13 second bound}\\
2(a_{12})^2 + 2(a_{13})^2 + (a_{23})^2 + 2a_{12}d + 2a_{13}d& > 3\alpha_{1} + 2\alpha_{2} + 2\alpha_{3}-3. \label{eq: a 23 second bound}
\end{align}	
\end{lemma}
\begin{proof}
For inequality~\eqref{eq: a 12 bound},  we bound the sum of the number of edges in colours $1$ and $2$. Clearly a pair of vertices from $V$ can contribute at most $2$ to the sum $\vert G''_1\vert +\vert G''_2\vert$. However by Proposition~\ref{prop: easy case, good structure}(iii) and (iv), pairs of vertices from $\left(V_{13}\right)^{(2)}$, $\left(V_{23}\right)^{(2)}$ and $D^{(2)}$ contribute at most $1$ to this sum. Further,  by Proposition~\ref{prop: easy case, good structure}(iv), a vertex-pair $xx'$ with $x$, $x'$ coming from two different sets $V_{ij}$ can contribute at most $1$ to this sum.  Finally, by Proposition~\ref{prop: easy case, good structure}(v), each edge from $M_{13}$ or $M_{23}$ sends at most two edges in colours $1$ or $2$ to a vertex $y\in D$, while each edge of $M_{12}$ sends at most three edges in colours $1$ or $2$ to a vertex $y\in D$. Summing over all such edges, we see that the total contribution to $\vert G''_1\vert +\vert G''_2\vert$ from vertex pairs $xy$ with $x\in V_{13}\cup V_{23}$ and $y\in D$ is at most $\left(\vert V_{13}\vert + \vert V_{23}\vert \right) \cdot \vert D\vert$, while the contribution from pairs $xy$ with $x\in V_{12}$ and $y\in D$ is at most $\frac{3}{2} \vert V_{12} \vert \cdot \vert D\vert$. It follows from this analysis that
\[ \vert G''_1\vert +\vert G''_2\vert \leq \binom{N}{2} + \binom{\vert V_{12}\vert}{2}+\frac{1}{2}\vert V_{12}\vert \cdot \vert D\vert = \binom{N}{2} +\binom{a_{12}N}{2}+\frac{1}{2}a_{12}dN^2 < \binom{N}{2}\left(1 + a_{12}(a_{12}+d)\right)+N.\]
Combining this upper bound with the lower bound for $\vert G''_1\vert +\vert G''_2\vert $ from Proposition~\ref{prop: easy case, good structure}(i), subtracting $N$ from both sides and dividing through by $\binom{N}{2}$, we get the desired inequality~\eqref{eq: a 12 bound}. Inequalities~\eqref{eq: a 13 bound} and~\eqref{eq: a 23 bound} are obtained in the same way, mutatis mutandis.

Next we turn our attention to the proof of inequality~\eqref{eq sum of a ij squared bound}. This is done by bounding the number of edges in colours $1$, $2$ and $3$. We see that each pair $xx'$ contributes at most one to the sum $\sum_{i}\vert G''_{i}\vert	$, with two exceptions. If $x,x'\in V_{ij}$, then $xx'$ is a bi-chromatic edge and contributes $2$ to this sum. Finally, some pairs $x\in V_{ij}$, $y\in D$ may also contribute up to $2$ to this sum; we bound the contribution of those pairs by appealing to Proposition~\ref{prop: easy case, good structure}(v) which implies that for each pair $xx'$ from $M_{ij}$, the sum of the contributions from $xy$ and $x'y$  to $\sum_{i} \vert G''_i\vert $ is at most $3$. Summing over all $\vert M_{ij}\vert = \vert V_{ij}\vert /2$ pairs $xx'\in M_{ij}$, we get
\begin{align}\label{eq: intermediary}
\sum_i \vert G''_i\vert \leq \binom{N}{2}+ \sum_{ij}  \left(\binom{\vert V_{ij}\vert }{2}+\frac{1}{2}\vert M_{ij}\vert \cdot \vert D\vert \right)< \binom{N}{2}\left(1+\sum_{ij} \left( (a_{ij})^2    +a_{ij}d \right)\right)  +3N.  \end{align}
On the other hand, summing up the lower bounds for $\vert G''_i\vert + \vert G''_j\vert $  we get from Proposition~\ref{prop: easy case, good structure}(i) for all three pairs $ij\in [3]^{(2)}$, we have
\[2\sum_i \vert G''_i\vert \geq 2\left(\sum_i\alpha_i\right)\binom{N}{2} +6N \]
Now, 
$\sum_{i}\alpha_i= x^2+y^2+z^2+1$, so combining this lower bound with the upper bound in~\eqref{eq: intermediary}, we get the desired inequality~\eqref{eq sum of a ij squared bound}. 

Inequalities~\eqref{eq: a 12 second bound}, ~\eqref{eq: a 13 second bound} and~\eqref{eq: a 23 second bound} can be proved similarly.
For instance, \eqref{eq: a 12 second bound} follows by counting edges in $G''_1$ and $G''_2$ twice and edges in $G''_3$ three times, and analysing how many times different types of pairs can be counted in this sum. Inequalities~\eqref{eq: a 13 second bound} and~\eqref{eq: a 23 second bound} can be proved by counting similar linear combinations of the $\vert G''_i\vert$.
\end{proof}

\noindent We shall now derive a contradiction from the system of inequalities we have derived (which unfortunately requires a significant amount of careful calculations). To do so, we shall make use of the following simple fact.
\begin{proposition}
	\label{prop:SumOfSquares}Let $b_{0}$, $c_{0}$ and $s$ be given
	non-negative reals satisfying $c_0\leq b_0$ and $2b_0+c_0\leq s$. Then the expression $a^{2}+b^{2}+c^{2}$ attains its maximum value subject to the conditions $b\geq b_{0}$, $c\geq c_{0}$,
	$a+b+c=s$ and $a\geq b\geq c$ uniquely when $a=s-b_{0}-c_{0}$, $b=b_{0}$ and $c=c_{0}$. 
\end{proposition}
\begin{proof}
	Immediate from the convexity of the function $x\rightarrow x^{2}$. 
\end{proof}

\begin{lemma}\label{lemma: easy case, contradiction}
Suppose that $a_{12}$, $a_{13}$, $a_{23}$ and $d$ are non-negative real numbers satisfying inequalities~\eqref{eq: a 12 bound}--\eqref{eq: a 23 second bound}. Then we have  $a_{12} + a_{13} + a_{23} + d > 1$. 
\end{lemma}

\begin{proof}

Since $(\alpha_1, \alpha_2)$ is a good pair, we have by definition $\alpha_1\geq \alpha_2$ and $\alpha_2-\alpha_3=2x^2+z^2-1\geq 0$, and hence $\alpha_2\geq \alpha_3$. In particular, the right hand-side in the inequalities~\eqref{eq: a 12 bound}, \eqref{eq: a 13 bound} and \eqref{eq: a 23 bound} form a decreasing sequence.   On the other hand, for $d$ fixed, the expressions on the left hand-side of the inequalities inequalities~\eqref{eq: a 12 bound}, \eqref{eq: a 13 bound} and \eqref{eq: a 23 bound} are increasing functions of $a_{12}$, $a_{13}$ and $a_{23}$ respectively. Similarly the right-hand sides of the inequalities~\eqref{eq: a 12 second bound}, \eqref{eq: a 13 second bound} and \eqref{eq: a 23 second bound} form an increasing sequence, and for $d$ fixed, the expressions on the left hand side are increasing functions of $a_{12}$, $a_{13}$ and $a_{23}$ respectively. Since the inequality~\eqref{eq sum of a ij squared bound} is invariant under any permutation of $(a_{12}, a_{13}, a_{23})$, it  follows that we may permute the first three coordinates of $(a_{12}, a_{13}, a_{23}, d)$ to ensure $a_{12}\geq a_{13}\geq a_{23}$, while still satisfying our constraints and without decreasing the value of $a_{12} + a_{13} + a_{23} + d$.

We may thus assume $a_{12}\geq a_{13}\geq a_{23}$  in the remainder of the proof. With this assumption in hand,  some of our inequalities become superfluous. Moving forward in the proof, we relax~\eqref{eq: a 23 bound} to a non-strict inequality and only use~\eqref{eq: a 13 bound}, the relaxed inequality~\eqref{eq: a 23 bound}, \eqref{eq sum of a ij squared bound} and \eqref{eq: a 12 second bound}.

Suppose for the sake of contradiction that we have chosen non-negative real numbers $a_{ij}$ and $d$ so that $a_{12} + a_{13} + a_{23} + d \leq 1$ and the inequalities~\eqref{eq: a 13 bound}, \eqref{eq: a 23 bound}, \eqref{eq sum of a ij squared bound} and \eqref{eq: a 12 second bound} are satisfied. Given the value of $a_{13} + a_{23}$, we can increase the value of $a_{13}$ while decreasing $a_{23}$ without violating the inequalities~\eqref{eq: a 13 bound},  \eqref{eq sum of a ij squared bound} or \eqref{eq: a 12 second bound}, as long as the inequality~\eqref{eq: a 23 bound} remains satisfied and as long the inequality $a_{12} \geq a_{13}$ is still satisfied. This is evident from the symmetric role played by the variables $a_{13}$ and $a_{23}$ and the convexity of the expressions in~\eqref{eq sum of a ij squared bound} and~\eqref{eq: a 12 second bound}. 

Thus we may assume that either $a_{12} = a_{13}$ or the inequality~\eqref{eq: a 23 bound} is tight. First let us suppose that $a_{12} = a_{13}$ and $a_{12} + a_{13} + a_{23} + d \leq 1$. Then it follows that 
\[\frac{1}{2} \leq x^2 + y^2 \leq x^2 + y^2 + z^2 \leq \sum_{ij} a_{ij}\left(a_{ij} + d\right) \leq d(1-d) + \sum_{ij} a_{ij}^{2} \leq d(1-d) + \left(\frac{1-d}{2}\right)^2.\]
However, it is easy to check that the inequality $d(1-d) + \left(\frac{1-d}{2}\right)^2 \geq \frac{1}{2}$ is false for every $d \in [0,1]$, and hence we are done in this case. 

Hence we may suppose that the inequality~\eqref{eq: a 23 bound} is tight, i.e.\ that we have $(a_{23})^2 + da_{23} = z^2$. Hence it follows that
\begin{align}\label{eq: value of a23}
a_{23} = \frac{-d + \sqrt{d^2 + 4z^2}}{2}.
\end{align}
Let $\delta \geq 0$ be chosen so that $a_{13}^2 + da_{13} = y^2 + \delta$, and note that the non-negativity of $\delta$ is guaranteed by~\eqref{eq: a 13 bound}. Hence we have
\begin{align} \label{eq: value of a13}
a_{13} = \frac{-d + \sqrt{d^2 + 4y^2 + 4\delta}}{2}.
\end{align}
Combining this with~\eqref{eq: value of a23}, we can simplify the inequalities~\eqref{eq: a 12 second bound} and~\eqref{eq sum of a ij squared bound} to obtain the following lower bounds for $a_{12}$:
\begin{align}
a_{12} &> \sqrt{x^2 - 2\delta},  \label{eq: a12 first}
\\
(a_{12})^2 +da_{12} &> x^2 - \delta, \qquad \textrm{ implying } \quad a_{12}>\frac{-d+\sqrt{4x^2+d^2-4\delta}}{2}, \label{eq: a12 second}
\\
a_{12} &\geq a_{13} = \frac{-d+\sqrt{d^2+4y^2+4\delta}}{2}. \label{eq: a12 third}
\end{align}

We start by observing that we must have $d > 0$. Indeed, if $d = 0$, then~\eqref{eq: value of a23} and ~\eqref{eq: value of a13} imply that $a_{13} \geq y$ and $a_{23} = z$. Increasing the value of $a_{12}$ if necessary, we may assume that $a_{12} + a_{13} + a_{23} = 1$ without violating~\eqref{eq sum of a ij squared bound}. However, then Proposition~\ref{prop:SumOfSquares} implies that $a_{12}^2 + a_{13}^2 + a_{23}^2 \leq x^2 + y^2 + z^2$, which contradicts ~\eqref{eq sum of a ij squared bound}. Thus we must have $d > 0$. 

Next, we note that we may assume $\delta < \frac{x^2 - y^2}{2}$. 
\begin{claim}\label{claim: x^2-y^2/2 bigger than delta}
If $\delta\geq \frac{x^2 - y^2}{2}$, then $a_{12} + a_{13} + a_{23} + d > 1$.
\end{claim}
\begin{proof}
Suppose $\delta\geq \frac{x^2 - y^2}{2}$. Then~\eqref{eq: value of a13} implies that
\begin{align*}
a_{13} \geq \frac{-d + \sqrt{d^2 + 2(x^2 + y^2)}}{2}
\end{align*}
Since $(\alpha_1, \alpha_2)$ is a good pair, 
\[x^2 + y^2 = \alpha_1 \geq \frac{1+\tau ^2}{2} > \frac{1}{2},\]
from which we deduce that $a_{13} > \frac{-d + 1}{2}$. Thus~\eqref{eq: a12 third} implies that we also have $a_{12} > \frac{-d+1}{2}$, and hence we conclude that $a_{12} + a_{13} + a_{23} + d > 1$, as required. 
\end{proof}
Assuming from now on that $\delta < \frac{x^2 - y^2}{2}$, we make a useful observation on the value of $x$ before splitting our analysis into two cases, depending on which of the two inequalities~\eqref{eq: a12 first} and~\eqref{eq: a12 second} gives the best lower bound for $a_{12}$.
\begin{claim}\label{claim: x at least 1-2tau}
$x\geq 1-2\tau$.	
\end{claim}
\begin{proof}
Since 	$(\alpha_1, \alpha_2)$ is a good pair, we have
\begin{align*}
1\leq 2x^2 + z^2 \leq 2x^2+\left(\frac{1-x}{2}\right)^2.
\end{align*}
Solving the associated quadratic inequality and using the fact that $x\geq 0$ yields the claimed lower bound on $x$: $x\geq \frac{1+2\sqrt{7}}{9}=1-2\tau$.
\end{proof}

\noindent\textbf{Case 1: $0 \leq \delta \leq d\sqrt{x^2 + d^2} - d^2$.} Let us fix $d>0$, and define the function $f\left(\delta\right)=f_{x,y,z,d}(\delta)$ for $\delta \in \left[0, d\sqrt{x^2+d^2} - d^2\right]$ by setting 
\begin{align*}
f\left(\delta\right):=\sqrt{x^2-2\delta} + \frac{\sqrt{d^2+4y^2+4\delta} + \sqrt{d^2+4z^2}}{2},
\end{align*}
and observe that by~\eqref{eq: a12 first} we have $a_{12} + a_{13} + a_{23} + d \geq f(\delta)$. Thus our aim is to prove that the least value of $f$ on this interval is strictly greater than $1$. The derivative of $f$ can be written as 
\begin{align*}
f'\left(\delta\right) = \frac{x^2 - 4y^2 -d^2 - 6\delta}{\sqrt{\left(x^2-2\delta\right)\left(d^2+4y^2+4\delta\right)}\left(\sqrt{x^2-2\delta} + \sqrt{d^2+4y^2+4\delta}\right)}.
\end{align*}
In particular, there exists a constant $c = \frac{x^2-4y^2 - d^2}{6}$ so that $f$ is increasing on $\left[0,c\right]$ and $f$ is decreasing on $\left[c,d\sqrt{x^2+d^2}-d^2\right]$. Hence $f$ attains its smallest value when $\delta = 0$ or $\delta = d\sqrt{x^2+d^2}-d^2$ (note that $c$ may not belong to the interval $\left[0,d\sqrt{x^2+d^2}-d^2\right]$, but the conclusion still remains true). Since $f(0) = x + \sqrt{y^2 + \frac{d^2}{4}} + \sqrt{z^2 + \frac{d^2}{4}} > x + y + z = 1$, we may turn our attention to analysing $f\left(d\sqrt{x^2+d^2} - d^2\right)$. It is easy to check that we have 
\begin{align} \label{eq:lower bound for f}
f\left(d\sqrt{x^2+d^2} - d^2\right) = \sqrt{x^2 + d^2} -d + \frac{\sqrt{4y^2 + 4d\sqrt{x^2+d^2}-3d^2}+\sqrt{4z^2+d^2}}{2} .
\end{align}

Let us consider~\eqref{eq:lower bound for f} with $z$ and $d$ fixed, and varying $x$ and $y$ while keeping $x+y$ as constant. Set $s := x+y$, and note that by Claim~\ref{claim: x at least 1-2tau} we have $x\geq 1-2\tau$ and thus $s \geq \frac{1+x}{2} \geq 1 - \tau$. Rewriting~\eqref{eq:lower bound for f} as a function $g(x)=g_{s,z,d}(x)$ of $x$, we obtain  
\begin{align} \label{eq function g}
g(x):=f\left(d\sqrt{x^2+d^2} - d^2\right) = \sqrt{x^2+d^2} -d + \frac{\sqrt{4(x-s)^2 + 4d\sqrt{x^2+d^2}-3d^2}}{2} + \frac{\sqrt{4z^2+d^2}}{2} ,\notag
\end{align}
whose derivative is given by
\begin{align*} 
g'(x) = \frac{x}{\sqrt{x^2+d^2}} + \frac{2(x-s) + \frac{xd}{\sqrt{x^2+d^2}}}{\sqrt{4(x-s)^2+4d\sqrt{x^2+d^2}-3d^2}}.
\end{align*}
Our aim is to show that $g'(x)$ is positive for $1- 2\tau \leq x \leq 1$. We first note that $4d\sqrt{x^2+d^2} - 3d^2 > d^2$. Since $x - s < 0$, we obtain that 
\begin{align*}
g'(x) \geq \frac{x}{\sqrt{x^2+d^2}} - \frac{2(s-x)}{\sqrt{4(x-s)^2 + d^2}}.
\end{align*}
Thus it suffices to prove that $g'(x) > 0$ in order to deduce that $x^2\left(4(s-x)^2 + d^2\right) > 4(s-x)^2 \left(x^2 + d^2\right)$. This follows from the fact that $2(s-x) - x \leq 2(1-x) - x \leq 6\tau -1 < 0$. Hence it suffices to prove that $f\left(d\sqrt{x^2+d^2} - d^2\right) > 1$ when $x = 1 - 2\tau$ and $y + z = 2\tau$. 

We now substitute the value $x=1-2\tau$ into~\eqref{eq:lower bound for f} and set  $h(d)=h_{y,z}(d):=f\left(d\sqrt{(1-2\tau)^2+d^2} - d^2\right)$. By differentiating and using the facts that $y \leq 1-x=2\tau$, that $z \leq \frac{1-x}{2}\leq \tau$, and that $4\sqrt{x^2+d^2}+\frac{4d^2}{\sqrt{x^2+d^2}}\geq 6d$, we get
\begin{align*}
h'(d) &= \frac{d}{\sqrt{x^2+d^2}} + \frac{4\sqrt{x^2+d^2} + \frac{4d^2}{\sqrt{x^2+d^2}}-6d}{4\sqrt{4y^2+4d\sqrt{x^2+d^2}-3d^2}} + \frac{d}{2\sqrt{4z^2+d^2}} - 1
\\ &\geq \frac{d}{\sqrt{(1-2\tau)^2+d^2}} + \frac{4\sqrt{(1-2\tau)^2 + d^2} + \frac{4d^2}{\sqrt{(1-2\tau)^2+d^2}} - 6d}{4\sqrt{16\tau^2 + 4d\sqrt{(1-2\tau)^2 + d^2} - 3d^2}} + \frac{d}{2\sqrt{4\tau^2 + d^2}} - 1. 
\end{align*}
Let $k(d)$ denote the function on the right hand side of the inequality above. 
As shown in the Appendix (inequality~\eqref{eq1}), the function $k(d)$ is positive for $d \in \left[0,1\right]$. In particular, it follows that $h'(d)$ is positive for all $d \in \left[0,1\right]$, and hence $h(d)$ is increasing. Thus $h(d) > h(0) = x + y + z = 1$ for all $d > 0$, which implies that $f\left(d\sqrt{x^2+d^2} - d^2\right) > 1$. Hence $a_{12} + a_{13} + a_{23} + d > 1$ whenever $\delta \in \left[0, d\sqrt{x^2+d^2} - d^2\right]$. This concludes the proof in this case.

\noindent\textbf{Case 2: $\delta \in \left[d\sqrt{x^2+d^2} - d^2, \frac{x^2-y^2}{2}\right]$.} 
Let $\ell(\delta)=\ell_{x,y,z,d}(\delta)$ denote the  function given by
\begin{align}
\ell\left(\delta\right) := \frac{\sqrt{d^2+4x^2-4\delta} + \sqrt{d^2+4y^2+4\delta} + \sqrt{d^2 + 4z^2}}{2} - \frac{d}{2}. 
\end{align}
Note that~\eqref{eq: value of a13}, \eqref{eq: value of a23} and \eqref{eq: a12 second} imply that we have $a_{12} + a_{13} + a_{23} + d \geq \ell\left(\delta\right)$. The derivative of $\ell$ is given by
\begin{align*} 
\ell'(\delta) =  \frac{1}{\sqrt{d^2+4y^2+4\delta}} - \frac{1}{\sqrt{d^2 + 4x^2 - 4\delta}}. 
\end{align*}
Since $\delta \leq \frac{x^2- y^2}{2}$, it follows that $d^2 + 4x^2 - 4\delta \geq d^2 + 4y^2 + 4\delta$, and hence we have $\ell'(\delta) \geq 0$ for $\delta \leq \frac{x^2 - y^2}{2}$. Thus $\ell(\delta)$ attains its minimum on our interval $\left[d\sqrt{x^2+d^2} - d^2, \frac{x^2-y^2}{2}\right]$ when $\delta = d\sqrt{x^2+d^2} - d^2$. As the inequalities~\eqref{eq: a12 first} and~\eqref{eq: a12 second} give the same bound for $a_{12}$ when $\delta = d\sqrt{x^2+d^2} - d^2$, we conclude that $\ell\left(d\sqrt{x^2+d^2} - d^2\right) \geq f\left(d\sqrt{x^2+d^2} - d^2\right)>1$ (the latter strict inequality being proved in our analysis of Case 1). Thus in this case also we must have $a_{12} + a_{13} + a_{23} + d > 1$. Combined with Claim~\ref{claim: x^2-y^2/2 bigger than delta}, our case analysis proves Lemma~\ref{lemma: easy case, contradiction}.
\end{proof}
\noindent Now the conclusion Lemma~\ref{lemma: easy case, contradiction} contradicts the fact that we have $a_{12} + a_{13} + a_{23} + d = 1$; this contradiction shows no counterexample to Theorem~\ref{theorem: technical version, easy case} exists, concluding the proof of the theorem.
\end{proof}

\subsection{The $\mathbf{H}$-extremal region: the case $\alpha_1\geq \alpha_2 +(1-\sqrt{\alpha_2})^2$}

Given a $3$-colouring template $\mathbf{G}$ on $N$ vertices with $\vert G_1\vert \geq \max\left\{\vert G_2 \vert, \vert G_3 \vert\right\}$, we define the function
\[g(\mathbf{G}):= \vert G_1\vert + \vert G_2\vert + \vert G_3\vert -2\binom{N}{2}- 2\max\left\{\vert G_2\vert, \vert G_3\vert\right\}+ 2\sqrt{\binom{N}{2}\max\left\{\vert G_2\vert , \vert G_3\vert \right\}}.\]

\begin{theorem}\label{theorem: technical version, hard case}
There exists an absolute constant $C>5$ such that the following hold: 
 if $\mathbf{G}$ is a $3$-colouring template on $n$ vertices satisfying $\vert G_1\vert \geq \vert G_2\vert \geq \vert G_3 \vert$ and
	\begin{align}\label{eq: bound on sum Gi Gj Gk hard case}
	g(\mathbf{G}) \geq Cn,
	\end{align}
  then $\mathbf{G}$ contains a rainbow triangle.
\end{theorem}
\begin{remark}
Setting $\vert G_i\vert =\alpha_i \frac{n^2}{2}$ for $i\in [3]$ and assuming $\alpha_3\leq \alpha_2$, \eqref{eq: bound on sum Gi Gj Gk hard case}  implies after rearranging terms and dividing through by $\frac{n^2}{2}$ that $\alpha_3\geq 2-\alpha_1+\alpha_2-2\sqrt{\alpha_2}+\Omega(n^{-1})$, which up to the error term is exactly the bound we require in Theorem~\ref{theorem: forcing densities} part (b).	
\end{remark}

\begin{proof}
Let $C>0$ be a sufficiently large constant to be specified later. It will be convenient to give a name to the function of $\max\left\{\vert G_2\vert, \vert G_3\vert \right\}$ involved in the definition of $g(\mathbf{G})$. Set therefore $f_n: \ \mathbb{R}_{\geq 0}\rightarrow \mathbb{R}$ to be the function given by
\[f_n\left(x\right):=x-\sqrt{x{n \choose 2}}.\]
When $n$ is clear from context, we often omit the subscript $n$ and write $f$ for the function $f_n$.
\begin{proposition}\label{prop: f minimum}
	The function $f$ is strictly decreasing in the interval $[0, \frac{1}{4}\binom{n}{2}]$ and strictly increasing in  the interval $[\frac{1}{4} \binom{n}{2}, \binom{n}{2}]$. Its unique minimum in $[0,\binom{n}{2}]$ is $f(\frac{1}{4}\binom{n}{2})=-\frac{1}{4}\binom{n}{2}$.
\end{proposition}
\begin{proof}
Simple calculus.	
\end{proof}	
\noindent
Suppose Theorem~\ref{theorem: technical version, hard case} is false, and let $N$ be the least value of $n\geq 3$ for which there exists a Gallai $3$-colouring template satisfying
 the assumptions of Theorem~\ref{theorem: technical version, hard case}. Among such Gallai colouring templates, let $\mathbf{G}$ be one maximising the size of the largest colour class $\vert G_1\vert $.
In the next lemma, we show that the sizes of the vertex set and of the colour classes in this putative counterexample to Theorem~\ref{theorem: technical version, hard case} cannot be too small.
\begin{lemma}\label{lemma: case 2, G2 large} The following hold:
	\begin{enumerate}[(i)]
\item  $\sum_{i=1}^3 \vert G_i\vert >\frac{3}{2}\binom{N}{2} +CN$;
\item $N> 4C$;
		\item 	$\vert G_2\vert > \frac{1}{4}\binom{N}{2}+\frac{C}{2}N$.
	\end{enumerate}
\end{lemma}
\begin{proof}  By Proposition~\ref{prop: f minimum} and~\eqref{eq: bound on sum Gi Gj Gk hard case}, we have
	\begin{align*}
	\sum_{i=1}^3\vert G_i\vert &> 2\binom{N}{2}+2f(\vert G_2\vert) +CN\geq \frac{3}{2}\binom{N}{2}+CN,
	\end{align*}
	establishing (i). Further, by Theorem~\ref{theorem: average density forcing rainbow $K_3$}\footnote{Formally, Theorem~\ref{theorem: average density forcing rainbow $K_3$} is stated for $N$ sufficiently large. However in the case $r=3$ it is easy to see that the claimed bound holds for all $N\geq 3$. Indeed, if $\mathbf{G}$ is an $3$-colouring template on $N$ vertices that contains no rainbow triangle, then for any set of vertices $S$ of size $3$ we must have $\sum_{i=1}^3 \vert G_i[S]\vert \leq 6=2\binom{3}{2}$, whence $\sum_{i=1}^3,\vert G_i\vert \leq 2\binom{N}{2}$ by averaging.} we have $\sum_i\vert G_i\vert \leq 2\binom{N}{2}$, which implies $\frac{1}{2}\binom{N}{2}> CN$ and thus $N>4C$. This proves (ii). Finally, observe that (i) implies that
	\begin{align*}
	2\vert G_2\vert \geq \vert G_2\vert +\vert G_3\vert \geq \sum_{i=1}^3 \vert G_i\vert -\binom{N}{2}>\frac{1}{2}\binom{N}{2}+CN,
	\end{align*} 
	and hence $\vert G_2\vert> \frac{1}{4}\binom{N}{2}+\frac{CN}{2}$, proving (iii).
\end{proof}
\noindent Next, we use the maximality of $\vert G_1\vert$ and the minimality of $N$ to prove two key structural lemmas about the colour classes of $\mathbf{G}$.
\begin{lemma}\label{lemma: hard case, G1 contains G2 cup G3}
 $G_{2}\cup G_3\subseteq G_{1}$. 
\end{lemma}
\begin{proof}
We first show $G_2\subseteq G_1$ using a simple idea from~\cite{KeevashSaksSudakovVerstraete04}.
 Consider the $3$-colouring template $\mathbf{G}'$ with colour classes given by  $G'_1=G_1\cup G_2$, $G'_2=G_1\cap G_2$ and $G'_3=G_3$. It it easily checked that $\mathbf{G}'$ is also a Gallai colouring template, 
 and that $\sum_{i=1}^3\vert G'_i\vert=\sum_{i=1}^3 \vert G_i\vert$.

 Our aim is to prove that $\mathbf{G}'$ also satisfies~\eqref{eq: bound on sum Gi Gj Gk hard case}. By Lemma~\ref{lemma: case 2, G2 large}(i), we have
 	\begin{align*}
	2\max \left(\vert G'_2\vert, \vert G'_3\vert  \right)\geq \sum_{i=1}^3 \vert G'_i\vert -\binom{N}{2}=\sum_{i=1}^3 \vert G_i\vert -\binom{N}{2}>\frac{1}{2}\binom{N}{2}+CN,
 	\end{align*} 
whence $g'_2:= \max \left(\vert G'_2\vert, \vert G'_3\vert  \right)$ satisfies $g'_2\geq \frac{1}{4}\binom{N}{2}+\frac{C}{2}N$. Clearly $g'_2\leq \vert G_2\vert$, whence $f(g'_2)\leq f(\vert G_2\vert)$ by Proposition~\ref{prop: f minimum}. We thus have
 \begin{align*}
 \sum_{i=1}^3 \vert G'_i\vert =\sum_{i=1}^3 \vert G_i\vert > 2\binom{N}{2}+2 f(\vert G_2\vert) +CN\geq 2\binom{N}{2}+2 f(g'_2) +CN,
 \end{align*}
so that (after swapping colours $2$ and $3$ if necessary) $\mathbf{G}'$ is also a Gallai template on $N$ vertices satisfying the assumptions of Theorem~\ref{theorem: technical version, hard case}. Since $\mathbf{G}$ was chosen to maximise the size of the first colour class among such counterexamples to Theorem~\ref{theorem: technical version, hard case}, we have that $\vert G_1\vert \geq \vert G'_1\vert=\vert G_1\cup G_2\vert$. Thus $G_2\subseteq G_1$, as claimed.

That $G_3\subseteq G_1$ is proved by using a similar, albeit simpler argument (since now both sides of~\eqref{eq: bound on sum Gi Gj Gk hard case} are unchanged when we replace $G_1$ and $G_3$ by $G_1\cup G_3$ and $G_1\cap G_3$ respectively).
\end{proof}
\begin{lemma}\label{lemma: hard case, no rainbow edges}
	There are no rainbow edges in $\mathbf{G}$: $G_{1}\cap G_2\cap G_3=\emptyset$.
\end{lemma}
\begin{proof}
Suppose for a contradiction that $xx' \in G_1\cap G_2\cap G_3$. We shall show the subtemplate $\mathbf{G}'$ induced by $V\setminus\{x,x'\}$ is a smaller counterexample to Theorem~\ref{theorem: technical version, hard case}.

Observe that for every $y\in V\setminus \{x,x'\}$, if one of the edges $xy, x'y$ is bi-chromatic or rainbow, then the other edge must be missing from $\bigcup_{i=1}^3 G_i$ (as otherwise we have a rainbow triangle in $\mathbf{G}$).  In particular,  writing $R$ for the number of rainbow edges from $xx'$ to $V\setminus \{x,x'\}$ (which by our observation satisfies $R\leq N-2$), we have
\begin{align}
\sum_{i=1}^3 \vert G'_i\vert \geq \sum_{i=1}^3 \vert G_i\vert - 2(N-2)-3-R&> 2\binom{N}{2}+2f_N(\vert G_2 \vert) +CN- 2N-R+1\notag\\
&= 2\binom{N-2}{2} +C(N-2) +2f_N
(\vert G_2\vert ) + 2N-R +2C-5\label{eq: bound on sum of Gi'}\end{align}
Clearly, the size of the second largest colour class in $\mathbf{G}'$ is at most $\vert G_2\vert -R-1$ (since both $\vert G_2\vert $ and $\vert G_3\vert$ decreased by at least $R+1$ when we removed the rainbow edge $xx'$ and the $R$ rainbow edges from $xx'$ to $V\setminus\{x,x'\}$). Now, we have that
\begin{align}\label{eq: difference between the fs}
-R+ 2f_{N}(\vert G_2\vert)-2f_{N-2}(\vert G_2\vert -R-1)= R+2-2\sqrt{\binom{N}{2}\vert G_2\vert}\left(1- \sqrt{\frac{N^2-5N+6}{N^2-N}}\sqrt{1 -\frac{R+1}{\vert G_2\vert}}  \right).
\end{align}
Write $\vert G_2\vert =\alpha_2 \binom{N}{2}$ and $R+1=\rho (N-1)$. By Lemma~\ref{lemma: case 2, G2 large}(iii), we know $\alpha_2 \in [\frac{1}{4},1]$. Further, by our observation above $R+1\leq N-1$, whence $\rho\leq 1$. By a straightforward asymptotic analysis,
\begin{align}\label{eq: asymptotic analysis fo rhte f difference}
1 + \rho (n-1) -2\binom{n}{2}\sqrt{\alpha_2}\left( 1- \sqrt{\frac{1-5n^{-1}+6n^{-2}}{1-n^{-1}}} \sqrt{1-\frac{2\rho}{\alpha_2 n}} \right)& =\left(\rho -2\sqrt{\alpha_2}-\frac{\rho}{\sqrt{\alpha_2}}\right)n +O(1).
\end{align} 
For $\alpha_2\in [1/4, 1]$, and $\rho \in [0,1]$, we have  $\rho - \frac{\rho}{\sqrt{\alpha_2}} \geq 1 - \frac{1}{\sqrt{\alpha_2}}$. Since $\left(1 - \sqrt{\alpha_2}\right) \left(\sqrt{\alpha_2} - \frac{1}{2}\right) \geq 0$, it follows that $1 - 2\sqrt{\alpha_2} - \frac{1}{\sqrt{\alpha_2}} \geq -2$. Combining this fact with~\eqref{eq: bound on sum of Gi'}, \eqref{eq: difference between the fs} and~\eqref{eq: asymptotic analysis fo rhte f difference}, and picking $C>5$ sufficiently large to ensure that we can absorb the $O(1)$ term in~\eqref{eq: asymptotic analysis fo rhte f difference}  with the $2C-5$ term in~\eqref{eq: bound on sum of Gi'} (recall that $N>4C$ by Lemma~\ref{lemma: case 2, G2 large}(ii), so picking $C$ sufficiently large ensures $N$ itself can be made sufficiently large), we get
\begin{align*}
\sum_{i=1}^3 \vert G'_i\vert & > 2\binom{N-2}{2}+2f_{N-2}(\vert G_2\vert -R-1) +C(N-2).
\end{align*}
We are now done once we observe that if $g'_2$ is the size of the second largest colour class in $\mathbf{G}'$, then $f_{N-2}(g'_2)\leq f_{N-2}(\vert G_2\vert -R-1)$. Indeed, as we noted above, $g'_2\leq \vert G_2\vert - R-1$. On the other hand, note that all colour classes have lost at most $1+ 2(N-2)=2N-3$ edges when we removed $xx'$ from $V$. Thus by Lemma~\ref{lemma: case 2, G2 large}(iii)
\[g'_2\geq \vert G_2\vert - 2N+3 \geq \frac{1}{4}\binom{N}{2}+\frac{C}{2}N -2N+3>\frac{1}{4}\binom{N-2}{2}.\]
As $f_{N-2}$ is increasing for $x\geq \frac{1}{4}\binom{N-2}{2}$ (Proposition~\ref{prop: f minimum}), this last inequality implies $f_{N-2}(g'_2)\leq f_{N-2}(\vert G_2\vert -R-1)$. Thus $\mathbf{G}'$ is indeed a counterexample to Theorem~\ref{theorem: technical version, hard case}, contradicting the vertex minimality of $\mathbf{G}$.
\end{proof}
\begin{corollary}\label{cor: no 23 edges}
There are no bi-chromatic edges in colours $23$: $G_2\cap G_3=\emptyset$.
\end{corollary}
\begin{proof}
	Since there are no rainbow edges (Lemma~\ref{lemma: hard case, no rainbow edges}), this is immediate from the fact proved in Lemma~\ref{lemma: hard case, G1 contains G2 cup G3} that $G_2\cap G_3\subseteq G_1$.
\end{proof}
\noindent
We now consider a largest matching of bi-chromatic edges from $\mathbf{G}$. By Lemmas~\ref{lemma: hard case, G1 contains G2 cup G3} and~\ref{lemma: hard case, no rainbow edges}, this matching $M$ is the disjoint union of two matchings $M_{12}$ and $M_{13}$ of bi-chromatic edges in colours $12$ and $13$ respectively. Let $V_{12}$ and $V_{13}$ denote the collections of vertices contained in some edge of $M_{12}$ and $M_{13}$ respectively, and let $D:=V\setminus (V_{12}\sqcup V_{13})$.

We shall perform modifications of $\mathbf{G}$ in a sequential manner, to obtain a new colouring template $\mathbf{G}''$, which may contain some rainbow triangles, but is well-structured with respect to the partition $V=V_{12}\sqcup V_{13} \sqcup D$ while still satisfying a slightly weaker form of~\eqref{eq: bound on sum Gi Gj Gk hard case}. This will lead to the desired contradiction (Lemma~\ref{lemma: hard case, contradiction}).

Before we start modifying our colouring template, we shall make some observations about $\mathbf{G}$, introduce an auxiliary graph $A$ on $M$, and observe some elementary properties of $g$, our function of colouring templates, all of which we shall need to analyse our sequence of modifications of $\mathbf{G}$. 
\begin{proposition}\label{prop: basic obs about the partition} For $j \in \{2,3\}$, set $\{k\}:=\{2,3\}\setminus \{j\}$. Then the following hold:
\begin{enumerate}[(i)]
	\item  for any pair of distinct edges $X,X'\in M_{1j}$, if there is any edge in colour $k$ from $X$ to $X'$, we have that at least two edges from $X$ to $X'$ are missing from $G_1\cup G_2$; 
\item for any edge $X\in M_{1j}$ and any edge  $Y\in M_{1k}$, we have $\sum_{i=1}^3 \vert G_i[X, Y]\vert \leq 5$, with equality only attained if  $\vert G_2[X, Y]\vert=\vert G_{3}[X,Y]\vert =1$ and $\vert G_1[X,Y]\vert=3$;
\item if $X\in M_{1j}$ and $Y, Y'$ are distinct edges in $M_{1k}$ such that $\sum_{i=1}^3 \vert G_i[X, Y]\vert= \sum_{i=1}^3 \vert G_i[X, Y']\vert=5$, then $\vert G_k[Y,Y']\vert \leq 3$;
\item if $v\in D$ and $X\in M_{1j}$ are such that there exists a bi-chromatic edge in colour $1k$ from $v$ to $X$, then $\sum_{i=1}^3 \vert G_i[X,v] \vert =2$;
\item there is no bi-chromatic edge in $D$ (and in particular $D$ is an independent set in $G_2\cup G_3$);
\item for every $X\in M_{ij}$, there exists at most one $v\in D$ sending bi-chromatic edges to both endpoints of $X$.
\end{enumerate}
\end{proposition}
\begin{proof}
Parts (i)--(iv) are immediate from the fact $\mathbf{G}$ is rainbow $K_3$-free and simple case analysis. Parts (v)--(vi) follow from the maximality of the bi-chromatic matching $M$ and the fact that $G_2\cup G_3\subseteq G_1$.	
\end{proof}
\noindent Next, we define an auxiliary graph $A$ on the edges of the matching $M_{12}\sqcup M_{13}$ by including a pair $X,X' \in M_{1j}$ in $A$ if $\vert G_j[X, X']\vert \leq 3$ and a pair $X\in M_{12}$, $Y\in M_{13}$ if $\sum_{i=1}^3 \vert G_i[X, Y]\vert = 5$. Finally, we make some elementary observations about $g$.
\begin{proposition}\label{prop: if g small}
Let $\mathbf{G}'$ be a $3$-colouring template with $\vert G'_1\vert \geq \max\left\{\vert G'_2\vert, \vert G'_3\vert \right\}$. If $\max\left\{\vert G'_2\vert, \vert G'_3\vert \right\}\leq \frac{1}{4}\binom{N}{2}+N$, then $g(\mathbf{G}')\leq 2N$.	
\end{proposition}
\begin{proof}Since $\vert G_1\vert \leq \binom{N}{2}$, our assumption together with the bound $\sqrt{1+2x}\leq 1+x$ gives
\[g(\mathbf{G}')\leq -\binom{N}{2} + \binom{N}{2}\sqrt{1+\frac{8}{N-1}} \leq 2N.\]	
\end{proof}
\begin{proposition}\label{prop: changes to g}
	Suppose $\max\left(\vert G'_2\vert, \vert G'_3\vert \right)\geq \frac{1}{4}\binom{N}{2}+N$. Then the following hold:
	\begin{enumerate}[(i)]
		\item the value of $g(\mathbf{G}')$ does not decrease if we delete a bi-chromatic edge and add two edges in colour $1$;
		\item for $t\leq N$, the value of $g(\mathbf{G}')$ decreases by at most $t$ if we delete up to $t$ edges in colours $2$ or $3$. 
	\end{enumerate} 
\end{proposition}
\begin{proof}
For part (i), assume without loss of generality that $\vert G'_2\vert \geq \vert G'_3\vert $.	Replacing a bi-chromatic  edge in colour $13$ by two edges in colour $1$ does not change the value of $g$. If $\vert G'_3\vert = \vert G'_2\vert$, then similarly we do not change the value of $g$ by removing a bi-chromatic edge in colours $12$ and adding in two edges in colour $1$. On the other other hand, if $\vert G'_3\vert < \vert G'_2\vert $, then deleting a bi-chromatic edge in colours $12$ and adding in two edges in colour $1$ keeps $\sum_{i=1}^3 \vert G'_i\vert$ constant and strictly decreases $2f(\vert G'_2\vert)= 2\vert G'_2\vert -2\sqrt{\binom{N}{2}\vert G'_2\vert}$ (since $f=f_N(x)$ is increasing in the interval $[\frac{1}{4}\binom{N}{2}, \binom{N}{2}]$, as shown in Proposition~\ref{prop: f minimum}); thus $g(\mathbf{G}')=\sum_{i=1}^3\vert G'_i\vert -2f(\max(\vert G'_2\vert, \vert G'_3\vert))$ actually increases in this case.

The proof of part (ii) follows similarly, and is left as an exercise to the reader.
\end{proof}
\noindent
We are now ready to embark  upon our sequence of modifications of $\mathbf{G}$. Set $\mathbf{G}'=\mathbf{G}$.  Recall that initially $\vert G'_2\vert\geq \frac{1}{4}\binom{N}{2}+\frac{C}{2}N>\frac{1}{4}\binom{N}{2}+2N$ by Proposition~\ref{lemma: case 2, G2 large}. Also, initially $\mathbf{G}'$ has the two properties that it contains no rainbow edge  and that it satisfies $G'_2\sqcup G'_3\subseteq G'_1$, both of which will be preserved by our modifications. Note however that our modifications will \emph{not} preserve the property of being rainbow $K_3$-free. Also, if the value of $\max\{\vert G_2'\vert, \vert G_3'\vert\}$ ever becomes too small by dropping below $\frac{1}{4}\binom{N}{2}+N$, we shall immediately stop the modification process.

\noindent\textbf{Step 1: dealing with $D$.} We go through the edges of the matching $M_{12}$. For each such edge $X$, we go through the vertices of $D$. If $v\in D$ sends a bi-chromatic edge of colours $13$, then by Proposition~\ref{prop: basic obs about the partition}(iv), we can replace this bi-chromatic edge by two edges in colour $1$ from $v$ to $X$. By Proposition~\ref{prop: changes to g}(i), this does not decrease the value of $g$. If this change brings $\max(\vert G'_2\vert , \vert G'_3\vert)$ below $\frac{1}{4}\binom{N}{2}+N$, then we stop our procedure and output the colouring template $\mathbf{G}''=\mathbf{G}'$.

We then repeat the same procedure with colours $2$ and $3$ switching roles, i.e.\ replace bi-chromatic edges in colours $12$ from $D$ to edges of $M_{13}$ by pairs of edges in colour $1$ (and outputting $\mathbf{G}''=\mathbf{G}'$ if the size of the second largest colour class ever becomes too small). Throughout, the value of $g(\mathbf{G}')$ does not decreases.

Next, we sequentially go through the edges $M_{12}\cup M_{13}$. By Proposition~\ref{prop: basic obs about the partition}(vi), for each such edge $X\in   M_{1j}$, there is at most one vertex $v_X\in D$ such that $v_X$ sends two bi-chromatic edges in colours $1j$ to $X$. If such a vertex $v_X$ exists, then we delete one of the two edges in colour $j$ from $v_X$ to $X$.

If the size of the second largest colour class in $\mathbf{G}'$ does not become too small, then at the end of this sequence of operations we have deleted at most $N/2<N$ edges in colours $2$ and $3$, and so by Proposition~\ref{prop: changes to g}(ii) we have $g(\mathbf{G}')\geq g(\mathbf{G})-N$ by the end of this step. Further, $\mathbf{G}'$ now has the property that for $j\in \{2,3\}$ at most half of the edges from $M_{1j}$ to $D$ are bi-chromatic in colours $1j$, and the rest of those edges are in colour $1$ or absent from $\bigcup_{i=1}^3G'_i$.

\noindent\textbf{Step 2: cleaning inside the $V_{1j}$.}  We sequentially go through the pairs of distinct edges $X,X' \in M_{12}$. For each such pair, if there is one edge in colour $3$ between $X$ and $X'$ then we have that (a) $XX'$ is an edge in our our auxiliary graph $A$, and (b) there are at least two edges from $X$ to $X'$ which are missing in $G_1\cup G_2$ (by Proposition~\ref{prop: basic obs about the partition}(i)). We then delete this edge in colour $3$, and arbitrarily add in one of the at least two missing edges in colour $1$ between $X$ and $X'$. If  there are two edges in colour $3$ between $X$ and $X'$, then we replace them with the two missing edges from $G_1[X, X']$, one after the other. By Proposition~\ref{prop: changes to g}, this does not decrease the value of $g(\mathbf{G}')$. Note that there cannot be more than $2$ edges in colour $3$ between them as $G_3 \subseteq G_1$. If one of our changes brings $\max(\vert G'_2\vert , \vert G'_3\vert)$ below $\frac{1}{4}\binom{N}{2}+N$, then we stop our procedure and output the colouring template $\mathbf{G}''=\mathbf{G}'$.

We then repeat the same procedure with colours $2$ and $3$ switching roles, i.e.\ replace edges in colours $2$ inside $V_{13}$ by edges in colour $1$ (and outputting $\mathbf{G}''=\mathbf{G}'$ if the size of the second largest colour class ever becomes too small). Throughout, the value of $g(\mathbf{G}')$ does not decrease (and thus remains at least $g(\mathbf{G})-N$). If the size of the second largest colour class in $\mathbf{G}'$ does not become too small in the process, then when we are done with this sequence of operations we have that for $j\in \{2,3\}$ the set $V_{1j}$ only contains edges in colours $1$ or $j$ and for every edge $XX'\in A[V_{1j}]$, there is (still) at least one edge in $(X, X')^{(2)}$ missing in $G'_j$.

\noindent\textbf{Step 3: cleaning across $V_{12}\times V_{13}$.} Recall the auxiliary graph $A$ introduced after Proposition~\ref{prop: basic obs about the partition}. We sequentially go through the pairs $X\in M_{12}$, $Y\in M_{13}$ with $XY\notin A$. For each such pair, we have $\sum_{i=1}^4G_i[X,Y]\leq 4$. So we can sequentially delete edges from $X$ to $Y$ in colours $2$ or $3$, and replace them by edges from $X$ to $Y$ in colour $1$. If this change brings $\max(\vert G'_2\vert , \vert G'_3\vert)$ below $\frac{1}{4}\binom{N}{2}+N$, then we stop our procedure and output the colouring template $\mathbf{G}''=\mathbf{G}'$.  By Proposition~\ref{prop: changes to g}(i), this does not decrease the value of $g(\mathbf{G}')$.

Next, we turn our attention to the pairs $X\in M_{12}$, $Y\in M_{13}$ with $XY\in A$. It follows from Proposition~\ref{prop: basic obs about the partition}(iii) that for each $X\in M_{12}$, the collection of $Y\in M_{13}$ with $XY\in A$ forms a clique in $A$.  By a graph theoretic result of Aharoni et al~\cite[Lemma 2.2]{AharoniDeVosdelaMazaMontejanoSamalin20},  under such a condition on the neighbourhoods we have 
\begin{align*}
\vert A[M_{12},M_{13}]\vert \leq \vert A[M_{12}] \vert + \vert A[M_{13}] \vert + \frac{\vert M_{12}\vert + \vert M_{13} \vert}{2}
\end{align*}
For convenience, set $e_{12} = \vert A[M_{12}] \vert$, $e_{13} = \vert A[M_{13}] \vert$ and $e = \vert A[M_{12},M_{13}] \vert$.   We begin by moving $\min\{e_{12},e\}$ edges from $G_2[M_{12}, M_{13}]$ to $G_2[M_{12}]$ and $\min\{e_{13},e\}$ edges from $G_3[M_{12}, M_{13}]$ to $G_3[M_{13}]$ (adding edges in colour $1$ to preserve $G'_2\sqcup G'_3\subseteq G'_1$ if necessary). This clearly does not decrease the value of $g$. Next we go through the remaining edges in colours $2$ or $3$ in $(M_{12},M_{13})^{(2)}$ one after the other, and replace all but at most $e -\min(e,e_{12})-\min(e,e_{13})\leq \frac{\vert M_{12} \vert + \vert M_{13} \vert}{2}$ of them by edges in colour $1$.

 To be more precise, at each step of this subprocess we let $j\in \{2,3\}$ be the second largest colour class in $\mathbf{G}'$ and $k$ the third largest colour class. If there is in $(M_{12},M_{13})^{(2)}$ any edge $f$ of $G_j$ and at least one missing edge in $G_1$, then we remove the edge $f$ in colour $j$ from $\mathbf{G}'$ and replace it by an edge $f'$ in colour $1$; if this brings $\max(\vert G'_2\vert, \vert G'_3\vert )$ below $\frac{1}{4}\binom{N}{2}+N$,  then we stop our procedure and output the colouring template $\mathbf{G}''=\mathbf{G}'$.  Otherwise if there is in $(M_{12},M_{13})^{(2)}$ any edge $f$ of $G_k$ and any edge $f'$ missing from $G_1$, then we remove the edge $f$ in colour $k$ from $\mathbf{G}'$ and replace it by an edge in colour $1$. By Proposition~\ref{prop: changes to g}(i) this does not decrease $g$.

 When the subprocess ends, we have at most $\frac{\vert M_{12}\vert + \vert M_{13} \vert }{2}\leq N/4$ edges in colours $2$ or $3$ left between $X$ and $Y$, which we remove. By Proposition~\ref{prop: changes to g}(ii), deleting these edges reduces the value of $g$ by at most $\frac{N}{4}$. If this brings $\max(\vert G'_2\vert, \vert G'_3\vert )$ below $\frac{1}{4}\binom{N}{2}+N$, then we stop our procedure and output the colouring template $\mathbf{G}''=\mathbf{G}'$. Otherwise, we have decreased the value of $g$ by at most $N/4$ in total in this step, whence $g(\mathbf{G}')\geq g(\mathbf{G})-2N$, and $\mathbf{G}'$ has the following property for $j\in \{2,3\}$:
\begin{align}\label{eq: keyprop G'}
\textrm{at most half of the edges from $V_{1j}$ to $D$ are in $G'_j$, and all other edges of $G'_j$ lie inside $V_{1j}$}.
\end{align}
We set $\mathbf{G}''=\mathbf{G}'$ and terminate our modification procedure. We are now ready to bound $g(\mathbf{G}'')$ and obtain the desired contradiction.
\begin{lemma}\label{lemma: hard case, contradiction}
	$g(\mathbf{G}'')\leq 3N$.
\end{lemma}
\begin{proof}
If $\vert G''_2\vert \leq \frac{1}{4}\binom{N}{2}+N$, then our claim is immediate from Proposition~\ref{prop: if g small}. Otherwise, set $x_{1j}N:=\vert V_{1j} \vert$ for $j\in \{2,3\}$ and $dN:=\vert D\vert$. By~\eqref{eq: keyprop G'}, we have $\vert G''_j \vert \leq \binom{x_{1j}N}{2}+\frac{1}{2}x_{1j}dN^2$. Clearly $\vert G''_1\vert\leq  \binom{N}{2}$.  Assume without loss of generality that $\vert G''_2\vert\geq \vert G''_3\vert $. 

Now, the function $x\mapsto -x+2\sqrt{x\binom{N}{2}}$ is increasing in the interval $[0, \binom{N}{2}]$. It then follows from the bounds on the sizes of the colour classes above that, for a choice of the constant $C>0$ sufficiently large, we have 
\begin{align*}
g&(\mathbf{G}'')= \vert G''_1\vert -2\binom{N}{2} + \vert G''_3\vert -\vert G''_2\vert+2\sqrt{\vert G''_2\vert \binom{N}{2}}\\
&\leq -\binom{N}{2}+ \left(\binom{x_{13}N}{2}+\frac{x_{13}dN^2}{2}\right)-\left(\binom{x_{12}N}{2}+\frac{x_{12}dN^2}{2}\right)+2\sqrt{\left(\binom{x_{12}N}{2}+\frac{x_{12}dN^2}{2}\right)\binom{N}{2}}\\
&<\frac{N^2}{2}\left(-1+(1-d-x_{12})^2 -(x_{12})^2 +(1-2x_{12}-d)d +2\sqrt{x_{12}(x_{12}+d)}\right) + 3N\\
&=N^2\left(-d-2x_{12}+2\sqrt{x_{12}(x_{12}+d)}\right) + 3N \leq 3N.
\end{align*}
\end{proof}
\noindent
Since, as noted at the end of our modification procedure, $g(\mathbf{G})\leq g( \mathbf{G}'')+2N$, it follows from Lemma~\ref{lemma: hard case, contradiction} that $g(\mathbf{G})\leq 5N$, whence $\mathbf{G}$ fails to satisfy~\eqref{eq: bound on sum Gi Gj Gk hard case} (since $C$ was chosen so that $C>5$), a contradiction. Thus there is no counterexample to Theorem~\ref{theorem: technical version, hard case}, concluding our proof.\end{proof}

\subsection{Putting it together: proof of Theorem~\ref{theorem: forcing densities} and Corollary~\ref{cor: conjecture true}}

\begin{proof}[Proof of Theorem~\ref{theorem: forcing densities}]
For part (a), the statement (i) follows from the definition of $\mathcal{R}'_1$: $\alpha_2=x^2+(1-x-y)^2\geq 1-x^2=\alpha_3$. The statement (ii) follows directly from Theorem~\ref{theorem: technical version, easy case}, while the statement (iii) follows from Proposition~\ref{prop: construction analysis}. Similarly for part (b), the statement (i) follows from the definition of $\mathcal{R}_2$: $\alpha_3=2-\alpha_1-2\sqrt{\alpha_2}+\alpha_2$, which is at most $\alpha_2$ since $\alpha_1\geq 2-2\sqrt{\alpha_2}$. The statements (ii) and (iii) then follow from Theorem~\ref{theorem: technical version, hard case} and Proposition~\ref{prop: construction analysis}.

Thus the only task that remains is to establish part (c). Our goal is to show that if $\alpha_1\geq \alpha_2$ and $(\alpha_1, \alpha_2)\notin \mathcal{R}'_1\cup \mathcal{R}_2$, then  $(\alpha_1, \alpha_2, \alpha_2)$ is not a forcing triple. By Propositions~\ref{prop: trivial conditions}(a)--(d), 
we see that if $\alpha_1<\frac{1+\tau^2}{2}$, $\alpha_2<\frac{1}{4}$, $\alpha_1+\alpha_2<1$ or $(\alpha_1, \alpha_2)\in \mathcal{R}_1\setminus \mathcal{R}'_1$, then we are done. The only region this leaves uncovered consists of the $(\alpha_1, \alpha_2)$ with $\alpha_2\in [\frac{1}{4}, \frac{1}{2})$ and $\max\left\{1-\alpha_2,   1-2\sqrt{\alpha_2}+2\alpha_2\right\}\leq \alpha_1<2-2\sqrt{\alpha_2}$.
Consider any $\alpha_2\in [\frac{1}{4}, \frac{1}{2})$. Then for any $\alpha_1< 2-2\sqrt{\alpha_2}$, there exists $\varepsilon>0$ such that $\alpha_1<2-2\sqrt{\alpha_2 + \varepsilon}-\varepsilon$ and $\alpha_2 + \varepsilon < \frac{1}{2}$. Setting $a=\lfloor \sqrt{(\alpha_2+\varepsilon) n}\rfloor$, $b=\lfloor \frac{2\sqrt{(\alpha_2+\varepsilon)}-1}{2\sqrt{(\alpha_2+\varepsilon)}}n\rfloor$ and $c=n-a-b$, we have (by Proposition~\ref{prop: construction analysis} or a quick calculation) that the $n$-vertex Gallai $3$-colouring template $\mathbf{H}(a,b,c)$ has colour density
\[  \left(2-2\sqrt{\alpha_2+\varepsilon}, \alpha_2+\varepsilon, \alpha_2+\varepsilon\right)  +\left(O(n^{-1}), O(n^{-1}), O(n^{-1})\right),\] 
from which it follows that $(\alpha_1, \alpha_2, \alpha_2)$ is not a forcing triple (since $\alpha_1+\varepsilon<2-2\sqrt{\alpha_2+\varepsilon}$).
\end{proof}

\begin{proof}[Proof of Corollary~\ref{cor: conjecture true}]
We need to show that in any Gallai $n$-vertex $3$-colouring template with colour density vector $(\alpha_1, \alpha_2, \alpha_3)$ and $\alpha_1\geq \alpha_2\geq \alpha_3$,  $\prod_{i=1}^3\alpha_i\leq h(\upsilon)+o(1)$. By Theorem~\ref{theorem: forcing densities}, it suffices to show this for $(\alpha_1, \alpha_2)\in \mathcal{R}'_1\cup \mathcal{R}_2$.

If $(\alpha_1, \alpha_2)\in \mathcal{R}_2$, then by Theorem~\ref{theorem: forcing densities}(b)(ii)--(iii), it is enough to show that for any choice of $x,y\geq 0$ with $1-x-y\geq 0$, the Gallai $3$-colouring template $\mathbf{H}(\lfloor xn\rfloor,\lfloor yn\rfloor,n-\lfloor x n\rfloor -\lfloor yn\rfloor)$ satisfies $\prod_{i=1}^3 \vert H_i \vert \leq h(\upsilon)\binom{n}{2}^3+O(n^5)$. Now, by Proposition~\ref{prop: construction analysis}, $\prod_{i=1}^3 \vert H_i \vert = f_H(x,y)\binom{n}{2}^3 +O(n^5)$, where $f_H(x,y)$ is given by
\[f_H(x,y):= (1-2xy)x^2 ((1-x)^2+2xy).\]
For $x$ fixed and $y\in [0, 1-x]$, simple calculus tells us that
\[ f_H(x,y)\leq \left\{ \begin{array}{ll}  f_H(x, \frac{2-x}{4}) & \textrm{if }x\leq\frac{2}{3}\\
f_H(x, 1-x) & \textrm{if }x>\frac{2}{3}.\end{array}\right.\]
Now it can be checked that $f_H(x, \frac{2-x}{4})$ is an increasing function of $x$ in the interval $[0, \frac{2}{3}]$ (its derivative with respect to $x$ is $\frac{x}{2}\left((1-x)^2+1\right)\left(3(x-\frac{2}{3})^2+\frac{2}{3}\right)\geq 0$), whence for any such $x$ we have $f_H(x, \frac{2-x}{4})\leq f_H(\frac{2}{3}, \frac{1}{3})$. It follows that for any choice of $x,y\geq 0$ with $x+y \leq 1$, we have  $f_H(x,y)\leq \max_{x'\in [0,1]}f_H(x',1-x')= \max_{x'\in [0,1]}\left((x')^2+(1-x')^2\right)(x')^2\left(1-(x')^2\right) = h(\upsilon)$ as required. 

Similarly if $(\alpha_1, \alpha_2)\in \mathcal{R}'_1$, then by Theorem~\ref{theorem: forcing densities}(a)(ii)--(iii), it is enough to show that for any choice of $x\geq \frac{1}{2}$ and $y\geq \frac{1-x}{2}$ with $1-x-y\geq 0$,
 the Gallai $3$-colouring template $\mathbf{F}(\lfloor xn\rfloor,\lfloor yn\rfloor,n-\lfloor x n\rfloor -\lfloor yn\rfloor)$ satisfies $\prod_{i=1}^3 \vert F_i \vert \leq h(\upsilon)\binom{n}{2}^3+O(n^5)$. By Proposition~\ref{prop: construction analysis} we have $\prod_{i=1}^3 \vert F_i\vert= f_F(x,y)\binom{n}{2}^3+O(n^5)$, where $f_F(x,y):=(x^2+y^2)\left(x^2+(1-x-y)^2\right)(1-x^2)$.  
Now 
\[\frac{\partial f_F}{\partial y}(x,y)= 4(1-x^2)\left(y-\frac{1-x}{2}\right) \left(x^2+y^2- (1-x)y\right),\]
and for  $y\geq \frac{1-x}{2}$ we have $x^2+y^2-(1-x)y \geq x^2+\left(\frac{1-x}{2}\right)^2 -2\left(\frac{1-x}{2}\right)^2= \frac{1}{4}(3x-1)(x+1)\geq 0$ for all $x\geq \frac{1}{2}$. Thus $\frac{\partial f_F}{\partial y}(x,y)\geq 0$ for $(x,y)$ in the domain we are considering, and $f_F(x,y)\leq f_F(x,1-x)= \left(x^2+(1-x)^2\right)x^2\left(1-x^2\right) \leq h(\upsilon)$, as required. 
\end{proof}

\nocite{*}

\section*{Acknowledgements}
The authors would like to express their gratitude to an anonymous referee for their careful work on the paper, which helped improved its correctness and clarity. Research on this paper was supported by the Swedish Research Council grant VR 2021-03687 and postdoctoral grant 213-0204 from the Olle Engkvist Foundation.

\appendix
\section{Appendix: verifying inequality (\ref{eq1})}
\label{inequality} 

We wish to prove the following inequality:
\begin{equation}\label{eq1}
k(d)=\frac{d}{2 \sqrt{d^2+4 \tau^2}}+\frac{\frac{4 d^2}{\sqrt{d^2+(1-2 \tau)^2}}+4 \sqrt{d^2+(1-2 \tau)^2}-6 d}{4 \sqrt{4 d \sqrt{d^2+(1-2 \tau)^2}-3 d^2+16 \tau^2}}+\frac{d}{\sqrt{d^2+(1-2 \tau)^2}}-1 \geq 0
\end{equation}
for $0\leq d \leq 1$  and  $\tau=\frac{4-\sqrt{7}}{9}$. A plot of the function $k(d)$ for $d\in[0,1]$ is provided in Figure~\ref{figure: inequality}, which may help a reader convince themselves the inequality is true. We give a rigorous proof below.

\begin{figure}\label{figure: inequality}
	\centering
		\includegraphics[scale=0.8]{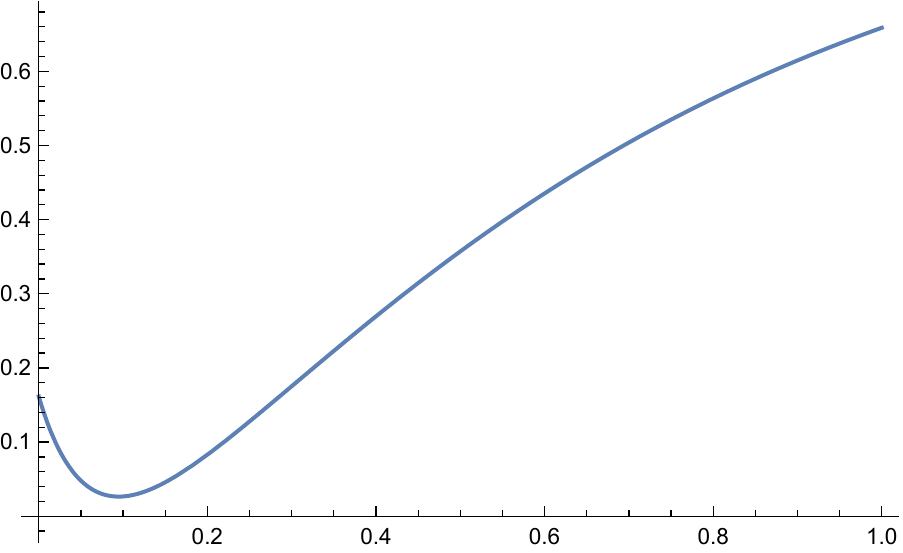}
	\caption{A plot of the function $k(d)$ from~\eqref{eq1} for $d\in [0,1]$}
\end{figure}

 Our first step is to find an upper bound on the modulus of the derivative of $k(d)$.  Differentiating term by term gives us that $k'(d)$ is equal to
\begin{multline*}
-\frac{d^2}{2 \left(d^2+4 \tau^2\right)^{3/2}}+\frac{1}{2 \sqrt{d^2+4 \tau^2}}+\frac{-\frac{4 d^3}{\left(d^2+(1-2 \tau)^2\right)^{3/2}}+\frac{12 d}{\sqrt{d^2+(1-2 \tau)^2}}-6}{4 \sqrt{4 d \sqrt{d^2+(1-2 \tau)^2}-3
		d^2+16 \tau^2}}\\-\frac{\left(\frac{4 d^2}{\sqrt{d^2+(1-2 \tau)^2}}+4 \sqrt{d^2+(1-2 \tau)^2}-6 d\right)^2}{8 \left(4 d \sqrt{d^2+(1-2 \tau)^2}-3 d^2+16 \tau^2\right)^{3/2}}-\frac{d^2}{\left(d^2+(1-2
	\tau)^2\right)^{3/2}}+\frac{1}{\sqrt{d^2+(1-2 \tau)^2}}.
\end{multline*}

\noindent Taking the modulus and distributing over the terms gives us an upper bound for $\vert k'(d)\vert$ of 

\begin{multline*}
\frac{d^2}{2 \left(d^2+4 \tau^2\right)^{3/2}}+\frac{1}{2 \sqrt{d^2+4 \tau^2}}+\frac{\frac{4 d^3}{\left(d^2+(1-2 \tau)^2\right)^{3/2}}+\frac{12 d}{\sqrt{d^2+(1-2 \tau)^2}}+6}{4 \sqrt{4 d \sqrt{d^2+(1-2 \tau)^2}-3 d^2+16
		\tau^2}}\\+\frac{\left(\frac{4 d^2}{\sqrt{d^2+(1-2 \tau)^2}}+4 \sqrt{d^2+(1-2 \tau)^2}+6 d\right)^2}{8 \left(4 d \sqrt{d^2+(1-2 \tau)^2}-3 d^2+16 \tau^2\right)^{3/2}}+\frac{d^2}{\left(d^2+(1-2
	\tau)^2\right)^{3/2}}+\frac{1}{\sqrt{d^2+(1-2 \tau)^2}}.
\end{multline*}

\noindent We upper bound this by setting $d=1$ where that maximizes a term, which yields

\begin{multline*}
\frac{\left(\frac{4}{\sqrt{d^2+(1-2 \tau)^2}}+4 \sqrt{(1-2 \tau)^2+1}+6\right)^2}{8 \left(4 d \sqrt{d^2+(1-2 \tau)^2}-3 d^2+16 \tau^2\right)^{3/2}}+\frac{1}{2 \sqrt{d^2+4 \tau^2}}+\frac{\frac{12}{\sqrt{d^2+(1-2
			\tau)^2}}+\frac{4}{\left(d^2+(1-2 \tau)^2\right)^{3/2}}+6}{4 \sqrt{4 d \sqrt{d^2+(1-2 \tau)^2}-3 d^2+16 \tau^2}}+\\\frac{1}{2 \left(d^2+4 \tau^2\right)^{3/2}}+\frac{1}{\sqrt{d^2+(1-2 \tau)^2}}+\frac{1}{\left(d^2+(1-2
	\tau)^2\right)^{3/2}}.
\end{multline*}

\noindent Next we set $d=0$ wherever that obviously maximizes a term, and use the fact that $\tau<\frac{1}{2}$ to substitute $1-2\tau$ for $\sqrt{(1-2\tau)^2}$, which gives
\begin{multline*}
\frac{\left(4 \sqrt{(1-2 \tau)^2+1}+\frac{4}{1-2 \tau}+6\right)^2}{8 \left(4 d \sqrt{d^2+(1-2 \tau)^2}-3 d^2+16 \tau^2\right)^{3/2}}+\frac{\frac{12}{1-2 \tau}+\frac{4}{\left(1-2
		\tau\right)^{3}}+6}{4 \sqrt{4 d \sqrt{d^2+(1-2 \tau)^2}-3 d^2+16 \tau^2}}+\\\frac{1}{4 \tau}+\frac{1}{16 \tau^3}+\frac{1}{1-2 \tau}+\frac{1}{\left(1-2 \tau\right)^{3}}.
\end{multline*}

\noindent Next, looking at $4 d \sqrt{d^2+(1-2 \tau)^2}-3 d^2+16 \tau^2$  we see that it is at least $d^2+16\tau^2>0$. What is more, its derivative with respect to $d$ is clearly greater or equal to $2d$.
Hence,  for $d$ in our interval $[0,1]$, this function is minimized at $d=0$.  So we set $d=0$ in the remaining expressions  and get 

\begin{multline*}
\frac{\left(4 \sqrt{(1-2 \tau)^2+1}+\frac{4}{1-2 \tau}+6\right)^2}{512 \tau^3}+\frac{\frac{12}{1-2 \tau}+\frac{4}{\left(1-2 \tau\right)^{3}}+6}{16 \tau}+\frac{1}{4\tau}+\frac{1}{16 \tau^3}+\frac{1}{1-2 \tau}+\frac{1}{\left(1-2 \tau\right)^{3}}.
\end{multline*}

\noindent Simplifying this we get 

\begin{equation*}
\frac{66716+31943 \sqrt{7}+12 \sqrt{19825442+7493276 \sqrt{7}}}{1152}  \approx 196.868  \leq 200.
\end{equation*}

\noindent Finally, evaluating the left hand side of inequality~\eqref{eq1} at 8000 evenly spaced points in the interval $[0,1]$ and using the fact the minimum cannot differ by more than $\frac{200}{8000}$ from the minimum of  these values we find that the left hand side of~\eqref{eq1} is bounded from below by $0.00147$. The actual minimum is $0.0264741$, which is achieved at $d \approx 0.0948007$.

\end{document}